\documentclass[11pt]{amsart}

\usepackage{amsfonts,epsfig}
\usepackage{latexsym}
\usepackage{amssymb}
\usepackage{amsmath}
\usepackage{amsthm}
\usepackage{graphics}
\usepackage{verbatim}
\usepackage{enumerate}
\usepackage{datetime}
\usepackage[all]{xy}
\usepackage{multirow}
\usepackage{array,booktabs,ctable}
\usepackage[backref]{hyperref}
\usepackage{color}
\definecolor{mylinkcolor}{rgb}{0.8,0,0}
\definecolor{myurlcolor}{rgb}{0,0,0.8}
\definecolor{mycitecolor}{rgb}{0,0,0.8}
\hypersetup{colorlinks=true,urlcolor=myurlcolor,citecolor=mycitecolor,linkcolor=mylinkcolor,linktoc=page,breaklinks=true}
\usepackage{bigstrut}
\usepackage{fancyhdr}

\newtheorem{defn}{Definition}[section]

\newtheorem{corollary}[defn]{Corollary}
\newtheorem{lemma}[defn]{Lemma}

\newtheorem{thm}[defn]{Theorem}
\newtheorem{theorem}[defn]{Theorem}
\newtheorem{cor}[defn]{Corollary}
\newtheorem{prop}[defn]{Proposition}

\newtheorem{conj}[defn]{Conjecture}

\theoremstyle{definition}

\newtheorem*{ack}{Acknowledgements}
\newtheorem{remark}[defn]{Remark}
\newtheorem{question}[defn]{Question}
\newtheorem{example}[defn]{Example}

\newcommand{\Q}{\mathbb Q}
\newcommand{\Z}{\mathbb Z}
\newcommand{\QQ}{\mathbb Q}

\newcommand{\ZZ}{\mathbb Z}
\renewcommand{\NN}{\mathbb N}
\newcommand{\FF}{\mathbb F}

\newcommand{\Rats}{\mathbb{Q}}

\newcommand{\SL}{\operatorname{SL}}

\newcommand{\Ker}{\operatorname{Ker}}

\newcommand{\Gal}{\operatorname{Gal}}
\newcommand{\Aut}{\operatorname{Aut}}
\newcommand{\GQ}{\Gal(\overline{\Rats}/\Rats)}

\newcommand{\GL}{\operatorname{GL}}
\newcommand{\PGL}{\operatorname{PGL}}

\newcommand{\tor}{\mathrm{tors}}

\renewcommand{\Im}{{\rm Im}\,}

\begin{document}

\title[Coincidences of division fields]{Coincidences of division fields}

%\fancyhf{}
%\fancyhead[L]{\jobname\ -- \today\  --\ \currenttime}
%\fancyhead[R]{\thepage}
%\pagestyle{fancy}

\author{Harris B. Daniels}
\address{Department of Mathematics, Amherst College, Amherst, MA 01002}
\email{hdaniels@amherst.edu}
\urladdr{www3.amherst.edu/$\sim$hdaniels/} 

\author{\'Alvaro Lozano-Robledo}
\address{Department of Mathematics, University of Connecticut, Storrs, CT 06269}
\email{alvaro.lozano-robledo@uconn.edu}
\urladdr{www.alozano.clas.uconn.edu/} 

\maketitle

\begin{abstract}
	Let $E$ be an elliptic curve defined over $\Q$, and let $\rho_E\colon \Gal(\overline{\Q}/\Q)\to \GL(2,\widehat{\Z})$ be the adelic representation associated to the natural action of Galois on the torsion points of $E(\overline{\Q})$. By a theorem of Serre, the image of $\rho_{E}$ is open, but the image is always of index at least $2$ in $\GL(2,\widehat{\Z})$ due to a certain quadratic entanglement amongst division fields. In this paper, we study other types of abelian entanglements. More concretely, we classify the elliptic curves $E/\Q$, and primes $p$ and $q$ such that $\Q(E[p])\cap \Q(\zeta_{q^k})$ is non-trivial, and determine the degree of the coincidence. As a consequence, we classify all elliptic curves $E/\Q$ and integers $m,n$ such that the $m$-th and $n$-th division fields coincide, i.e., when $\Q(E[n])=\Q(E[m])$, when the division field is abelian.
\end{abstract}

\section{Introduction}

Let $E/\QQ$ be an elliptic curve, let $n>1$ be an integer, let $\overline{\Q}$ be a fixed algebraic closure of $\Q$, and let $\Q(E[n])\subset\overline{\Q}$ be the $n$-th division field, i.e., $\Q(E[n])$ is the field of definition of the $n$-torsion subgroup $E[n]\subseteq E(\overline{\Q})$. The absolute Galois group of $\Q$, hereby denoted by $G_\Q=\Gal(\overline{\Q}/\Q)$, acts on $E[n]$ and induces a Galois representation $\rho_{E,n}\colon G_\Q \to \Aut(E[n])\simeq \GL(2,\ZZ/n\ZZ)$. If $p$ is a prime, then $G_\Q$ acts on the Tate $p$-adic module $T_p(E)=\varprojlim E[p^n]$, and on the  $\widehat{\ZZ}$-module $T(E)=\varprojlim E[n]$, and induces $p$-adic representations $\rho_{E,p^{\infty}}\colon G_\Q \to \GL(2,\ZZ_p)$, and an adelic Galois representation $$\rho_E \colon G_\Q \to \Aut(T(E))\simeq \GL(2,\widehat{\ZZ}).$$
There has been much recent work and interest in understanding the image of the various Galois representations mentioned above (see for example \cite{serre1, RZB, zywina1, zywina2, Sutherland2}). 

Famously, Serre \cite{serre1} showed that if $E/\QQ$ has no complex multiplication, then the image $G_E$ of $\rho_{E}$ is open (therefore, of finite index) in $\GL(2,\widehat{\ZZ})$. Further, it is well-known (pointed out by Serre in \cite{serre1}, Proposition 22) that the index $d_E=[\GL(2,\widehat{\ZZ}):G_E]\geq 2$. Indeed, if $\Delta_E$ is the minimal discriminant of $E/\Q$, then $\Q(\sqrt{\Delta_E})\subseteq \Q(E[2])$ and there is also some $m>2$ (the integer $m=4|\Delta_E|$ works) such that $\Q(\sqrt{\Delta_E})\subseteq \Q(\zeta_m)\subseteq \Q(E[m])$, so that $\Q(E[2])\cap \Q(\zeta_m)$ is a non-trivial quadratic extension of $\Q$, and therefore $\Gal(\Q(E[2],\zeta_m))\subsetneq \Gal(\Q(E[2])/\Q)\times \Gal(\Q(\zeta_m)/\Q).$ This {\it entanglement} of division fields causes the index $d_E$ to be at least $2$. When the index $d_E$ is exactly $2$, then we say that $E$ is a Serre curve, and these have been studied in \cite{Jones, Jones3, Daniels}, for instance. 

It is therefore natural to study other types of entanglements of division fields that would cause $d_E$ to be strictly larger than $2$. For instance, Brau and Jones \cite{brau}, and Morrow \cite[Theorem 8.7]{morrow} have classified all elliptic curves $E/\Q$ such that $\Q(E[2])\subseteq \Q(E[3])$ (in fact, their work classifies all the possible non-abelian fields that can occur as $\Q(E[2])\cap\Q(E[3])$). In this paper, we consider the following question:

\begin{question}\label{ques:main}
Fix an elliptic curve $E/\QQ$, and distinct integers $n,m\geq 2$:
\begin{enumerate}
	\item  Are there distinct integers $n,m\geq 2$ such that $\QQ(E[n]) = \QQ(E[m])$? 
	\item  In light of the entanglement described above, are there distinct prime numbers $p$ and $q$, and $k\geq 1$, such that $\Q(E[p])\cap \Q(\zeta_{q^k})$ is non-trivial?
	\end{enumerate}  
If so, can we classify all the elliptic curves $E$ for which (1) or (2) occurs? Note that (1) can be interpreted vertically (in towers, i.e., $n$ divides $m$) or horizontally ($\gcd(n,m)=1$). We will address both possibilities. 
\end{question}

There has been prior work on abelian entanglements related to Question \ref{ques:main}, part (2). In \cite{LR+GJ}, Gonz\'alez-Jim\'enez and the second author classified all elliptic curves such that the full $n$-th division field $\Q(E[n])$ is an abelian extension of $\Q$. More generally,  Chou \cite{chou} has classified the torsion subgroups $E(\Q^{\text{ab}})_{\text{tors}}$ that can occur for elliptic curves $E$ over $\Q$, where $\Q^{\text{ab}}$ is the maximal abelian extension of $\Q$ within a fixed algebraic closure. Here we shall extend these works by studying the possibilities for $\Q(E[p])\cap \Q^{\text{ab}}$.

It is worth noting that, by results of \cite{Duke,Jones}, almost all elliptic curves are Serre curves (that is, $d_E=2$). In particular, for almost all elliptic curves $E/\QQ$ we have that $\Gal(\QQ(E[n])/\QQ)\simeq \GL(2,\ZZ/n\ZZ)$, for all odd $n\geq 2$, and so comparing their degrees one can see that there are no $m\neq n\geq 2$ such that $\QQ(E[n]) = \QQ(E[m])$. Similarly, it follows that for a Serre curve $\Q(E[p])\cap \Q(\zeta_{q^k})$ is always trivial for all odd primes $p\neq q$, and all $k\geq 1$.
%Using the results of \cite{Duke,Jones} one can show, using a simple degree argument, that for ``most'' elliptic curves $E/\QQ$ there are no $m\neq n$ such that $\QQ(E[n]) = \QQ(E[m])$\footnote{The idea is that for ``most'' elliptic curves $\Gal(\QQ(E[n])\simeq \GL(2,\ZZ/n\ZZ)$ and if $m\neq n$ then $\GL(2,\ZZ/n\ZZ)\not\simeq \GL(2,\ZZ/m\ZZ)$ since they aren't even the same size. Thus, $\QQ(E[n])\neq\QQ(E[m])$. }. 
%ere, ``most'' has a technical definition that can be found in \cite{Jones}. Thus, if is possible for two distinct division fields for one elliptic curve to be equal, examples should be rare and because of this if $\QQ(E[n]) = \QQ(E[m])$ we say that $E$ has a coincidence of division fields. 
Thus, examples of coincidences of division fields should be somewhat rare. Nonetheless, with a simple search one can find some examples of such behavior. 

\begin{example}%Cremona Ref 486d2
Let $E$ be the elliptic curve with Cremona label \href{https://www.lmfdb.org/EllipticCurve/Q/486d2/}{\texttt{486d2}} which is given by $y^2=x^3+405x-9882$, and let $L$ be the splitting field of $x^3+3$, i.e., $L=\QQ(\zeta_3,\sqrt[3]{3})$, with $\zeta_3$ a primitive $3$-rd root of unity. Then, as we shall see below,
$$\QQ(E[2])=\QQ(E[3])=\QQ(E[6])=L.$$
\end{example}

\begin{example}\label{ex-2}%Cremona Ref 40a4
Let $E$ be the elliptic curve with Cremona label \href{https://www.lmfdb.org/EllipticCurve/Q/40a4/}{\texttt{40a4}} which is given by $y^2 = x^3+13x-34$. A simple computation shows that $$E(\QQ)_\tor=\langle (7,-20) \rangle\simeq\ZZ/4\ZZ.$$ Moreover, we have $x^3+13x-34=(x-2)(x^2+2x+17)$, where the discriminant of the quadratic factor is $-64$, and so $\QQ(E[2])=\QQ(i)$. Next, factoring the 4-division polynomial of $E$ we obtain
$$f_4(x) = 8 (x - 7)  (x - 2)  (x + 3)  (x^2 - 2x + 5)  (x^2 + 2x + 17)  (x^2 + 6x + 109)$$ 
which splits completely over $\QQ(i)$ and, further, one can verify that $\QQ(x(E[4]))=\QQ(E[4])$. Thus, 
$$E(\QQ(i))_\tor=\langle (7,-20),(-3-10i, 30+10i) \rangle\simeq \ZZ/4\ZZ\times\ZZ/4\ZZ.$$
Therefore, $\QQ(E[2])=\QQ(E[4])=\QQ(i)$. Stevenhagen asked when can one have $\QQ(E[2^{n}])=\QQ(E[2^{n+1}])$ and Rouse and Zureick-Brown  comment in a recent paper (see \cite[Remark 1.6]{RZB}) that, as a consequence of their work, this can only happen for $n=1$ (that is, for $\QQ(E[2])=\QQ(E[4])$), and give some examples.
\end{example}

Our first result addresses Question \ref{ques:main} in the setting of towers. 

\begin{theorem}\label{thm:main_vertical}
	Let $E$  be an elliptic curve defined over $\QQ$, let $p$ a prime, and let $n\in\NN$.
	
	\begin{enumerate}
		\item Suppose that $\QQ(E[p^{n+1}]) = \QQ(E[p^n])$. Then,  $p=2$, $n=1$, and there is a rational number $t\in\QQ$ such that $E$ is isomorphic over $\QQ$ to an elliptic curve of the form 
	$$E_t\colon y^2 = x^3 + A(t)x + B(t),$$ where 
	\begin{align*}
	A(t) &= -27t^8+648t^7-4212t^6-2376t^5+60102t^4+79704t^3-105732t^2-235224t-107811,\\
	B(t) &= 54t^{12}-1944t^{11}+24300t^{10}-97848t^9-251262t^8+1722384t^7+4821768t^6\\&\ \ \ -8697456t^5-64323558t^4-140447736t^3-157012020t^2-90561240t-21346578.
	\end{align*}
	
	\item If  $\Q(E[p^n])\cap \Q(\zeta_{p^{n+1}}) = \Q(\zeta_{p^{n+1}})$, then $p=2$.
	\end{enumerate}  
\end{theorem}

Theorem \ref{thm:main_vertical} will be shown in Section \ref{sec-towers}. Interestingly, $\Q(E[2^n])\cap \Q(\zeta_{2^{n+1}}) = \Q(\zeta_{2^{n+1}})$ can indeed occur for all $n>1$, as we will show at the end of Section \ref{sec:even_vert_CM} (see Theorem \ref{thm-2nroots}).

\begin{theorem}\label{thm-2nroots-intro}
	Let $E$ be the elliptic curve with Cremona label \href{https://www.lmfdb.org/EllipticCurve/Q/32a3/}{\texttt{32a3}}, which is given by $y^2=x^3-11x-14$. Then, $\Q(\zeta_{2^{n+1}})\subseteq \Q(E[2^n])$ for all $n>1$.
\end{theorem}

Our next results address Question \ref{ques:main} in a horizontal way. First, we remark that ramification and good reduction impose strong restrictions on equality of division fields.

\begin{prop}
	Let $E/\QQ$ be an elliptic curve and let $m\geq 2$ be an integer, such that there is a prime $p$ of good reduction for $E/\Q$ that does not divide $m$. Then,  $\Q(E[m])\cap \Q(\zeta_{p^{\infty}})=\Q$. In particular, if $n\geq 2$ is another integer divisible by $p$, then $\Q(E[n])=\Q(E[m])$ is impossible.
\end{prop}
\begin{proof}
	If $p$ is a prime of good reduction and $\gcd(p,m)=1$, then the criterion of N\'eron--Ogg--Shafarevich shows that $\Q(E[m])/\Q$ is unramified at $p$. Thus, $\Q(E[m])\cap \Q(\zeta_{p^{\infty}})=\Q$ by ramification-at-$p$ considerations. If in addition $p$ divides $n$, then $\Q(\zeta_{p})\subseteq \Q(E[p])\subseteq \Q(E[n])$, and the result follows.
\end{proof}

The following theorem answers the question for the intersection of prime division fields for two different primes. 

\begin{theorem}\label{thm:main_horizontal}
	Let $E/\QQ$ be an elliptic curve and let $p<q\in \ZZ$ be distinct primes, and let $n,m\in\NN$. 
	\begin{enumerate}
		\item If $\QQ(E[p^n]) = \QQ(E[q^m])$, then $p^n = 2$ and $q^m=3$. Further, there is some $t\in\QQ$ such that $E$ is $\QQ$-isomorphic to 
	\begin{align*}
	E'\colon y^2 = &x^3 -3 t^9 (t^3 - 2) (t^3 + 2)^3 (t^3 + 4)x\\ 
	&-2 t^{12} (t^3 + 2)^4 (t^4 - 2 t^3 + 4 t - 2) (t^8 + 2 t^7 + 4 t^6 + 8 t^5 + 10 t^4 + 8 t^3 + 16 t^2 + 8 t + 4)
	\end{align*}
	or its twist by $-3$.
	\item Let $K_p(E)=\Q(E[p])\cap \Q^{\text{ab}}$. Then, $\Gal(K_p(E)/\Q)\simeq (\Z/p\Z)^\times \times C$, where $C$ is a cyclic group of order dividing $p-1$. Further, if $E/\Q$ does not have a rational $p$-isogeny, then $C$ is trivial or quadratic and $K_p(E)=F(\zeta_p)$ with $F/\Q$ a trivial or quadratic extension. 
	\item In particular, if $\Q(\zeta_{q^n})\subseteq \Q(E[p])$, then either $\Q(\zeta_{q^n})=\Q$, $\Q(i)$, or $\Q(\zeta_3)$, or $E/\Q$ has a rational $p$-isogeny, $p=2,3,5,7,11,13,17,19,37,43,67,$ or $163$, and $\varphi(q^n)$ divides $p-1$. 
		\end{enumerate}
\end{theorem}

For example, the curve $E/\Q: y^2 + xy + y = x^3 - x^2 - 2x - 26$, with Cremona label \href{https://www.lmfdb.org/EllipticCurve/Q/405d1/}{\texttt{405d1}}, satisfies $\Q(\zeta_9)\subseteq \Q(E[7])$.  Finally, our third theorem deals with the particular case of abelian division fields. 

\begin{theorem}\label{thm:abelian}
	Let $E/\QQ$ be an elliptic curve and let $n>m\geq 2$ be integers, such that $\Q(E[n])/\Q$ is an abelian extension. 
	\begin{enumerate}
		\item If $\QQ(E[n]) = \QQ(E[m])$, then there are two possibilities.
		\begin{enumerate}
			\item Either 
		$m=2,n=4$, and for some $t\in\QQ$, $E/\QQ$ is $\QQ$-isomorphic to $$y^2 = x^3 + (-432t^8 + 1512t^4 - 27)x + (3456t^{12} + 28512t^8 - 7128t^4 - 54).$$
	In this case, $\QQ(E[2]) = \QQ(E[4]) = \QQ(i)$.
	\item Or $m=3,n=6$, with $\Q(E[2])\subsetneq \Q(E[3])=\Q(E[6])$, and there is some $t\in\QQ$, such that $j(E)=j(t)$ where
$$j(t)=-\left(\frac{(t^3 - 3t^2 - 9t - 9)(t^3 + 3t^2 + 3t - 3)(t^6 + 12t^5 + 81t^4 + 216t^3 + 243t^2 + 108t + 27)}{t(t+1)^2(t+3)^2(t^2+3)^2(t^2+3t+3)}\right)^3,$$
Conversely, if $E'/\Q$ is an elliptic curve such that $j(E')=j(t)$ for some $t\in\Q$, then there is a quadratic twist $E''/\Q$ of $E'$ such that $\Q(E''[2])\subsetneq \Q(E''[3])=\Q(E''[6])$.
\end{enumerate}
	\item Let $p$ be prime, such that $\Q(E[p])/\Q$ is abelian, and let $q\neq p$ be another prime. Then, $\Q(E[p])\cap \Q(\zeta_{q^k})$ can be trivial, quadratic, cyclic cubic (for $p=2$), or cyclic quartic (for $p=5$). 
\end{enumerate}
\end{theorem}

For example, let $E/\Q: y^2 = x^3 - x^2 - 4319x + 100435$, with Cremona label \href{https://www.lmfdb.org/EllipticCurve/Q/18176r2/}{\texttt{18176r2}}. Then, $\Gal(\Q(E[5])/\Q)\simeq (\Z/4\Z)^2$, and $\Q(E[5])=F(\zeta_5)$ is the compositum of $\Q(\zeta_5)$ and a cyclic quartic field $F\subseteq \Q(\zeta_{16})$.

It is worth pointing out that the family of elliptic curves that appears in part (1) of Theorem \ref{thm:main_vertical} is a parametrization of the modular curve $\tt X_{20b}$ from \cite{RZB}. Similarly, the family that appears in part (1) of Theorem \ref{thm:abelian} is $\tt X_{60d}$. Interestingly, $\tt X_{60d}$ parametrizes a subfamily of $\tt X_{20b}$ (see Remark \ref{rem-groupG} for more on this).  The family that appears in Theorem \ref{thm:main_horizontal} will be constructed in the proof of the theorem at the end of the paper. 

	Any computations done in this paper have been done using Magma \cite{Magma} and some code used in this paper was adapted from code written for \cite{Daniels, DLRNS, DDH,lozano1, SZ16}. For the ease of the reader, anytime a specific elliptic curve is mentioned, we refer to the curve by Cremona reference and include a link to the corresponding LMFDB \cite{lmfdb} page.

The structure of the paper is as follows. In Section \ref{sec:gal_reps} we introduce the necessary background about Galois representations.  Section \ref{sec-towers} contains the proof of Theorem \ref{thm:main_vertical}. We first prove the theorem for odd primes in Section \ref{sec:odd_vert}, in Section \ref{sec:even_vert} we use \cite{RZB} to settle the case of $p=2$ for non-CM curves, and in Section \ref{sec:even_vert_CM} we deal with the case of $p=2$ in the CM case using results from \cite{lozano1}. In Section \ref{sec-abelian} we show Theorem \ref{thm:abelian}, relying on results from \cite{LR+GJ}. In Section \ref{sec:abelian_exts} we examine the fields $\QQ(E[p])\cap\QQ^{ab}$ so that finally, in Section \ref{sec:horizontal} we can prove Theorem \ref{thm:main_horizontal}. 

\begin{remark}
	Unfortunately, our methods are not sufficient to classify all the instances when $\Q(E[m])=\Q(E[n])$ for any natural numbers $n>m\geq 2$. We suspect that the only possibilities are $(n,m)\in \{ (2,3),(2,4),(2,6),(3,6)\}$ but the current knowledge on the possible adelic images of $\rho_E\colon \Gal(\overline{\Q}/\Q)\to \GL(2,\widehat{\Z})$ is not sufficient to settle all the possible coincidences. For instance, to conclude a complete list, we would need to know the full list of possible mod-$p^2$ images, but at the moment this is not known (not even for $p=3$). More concretely, a more detailed understanding of the classification of mod-$9$ images would be necessary to rule out, for example, the pair $(6,9)$, that would be a coincidence between the $6$-th and the $9$-th division field. With a bit of computational help, we can prove the following result for coincidences in the range $2\leq m <n\leq 10$.
\end{remark}

\begin{thm}\label{thm-whatpairs}
	Let $2\leq m < n \leq 10$ be natural numbers, let $E/\Q$ be an elliptic curve, and suppose that $\Q(E[m])=\Q(E[n])$. Then,
	$$(m,n)\in \{(2,3),(2,4),(2,6),(3,6),(4,6),(6,8),(6,9),(5,10)\}.$$
\end{thm}

Our results can also show some other special cases, as the following corollary exemplifies (Corollary of Theorem \ref{thm:main_horizontal}).

\begin{cor}\label{thm-whatabout3}
	Let $p$ be a prime, and let $m\geq 2$ be an integer divisible by $q^n$ for some odd prime $q$ and $n\geq 1$, such that $\varphi(q^n)$ does not divide $p-1$. Then, $\Q(E[p])=\Q(E[m])$ is impossible. 
\end{cor}

\begin{example}
	As a consequence of Corollary \ref{thm-whatabout3}, the division fields $\Q(E[3])$ and $\Q(E[m])$ cannot coincide for any integer $m$ divisible by a prime $p\geq 5$.
\end{example}

The proofs of Theorem \ref{thm-whatpairs} and Corollary \ref{thm-whatabout3} can be found in Section \ref{sec-whatpairs}.

\begin{ack}
The authors would like to thank Enrique Gonz\'alez-Jim\'enez, Jackson Morrow, and Filip Najman for helpful comments on an earlier draft of this paper. We would also like to thank the referees for their many helpful comments and suggestions.
\end{ack}

\section{Galois Representations associated to Elliptic Curves}\label{sec:gal_reps}

 In this section we cite a number of key results that we will use in the following sections. Let $E/\Q$ be an elliptic curve. For a prime number $p$, we define the $p$-adic Tate module of $E/\QQ$ by  $T_p(E) = \varprojlim E[p^n]$, where the inverse limit is taken with respect to  the multiplication-by-$p$ maps $[p]\colon E[p^{n+1}]\to E[p^n]$. The absolute Galois group of $\QQ$ acts on $T_p(E)$, and induces a Galois representation
$$\rho_{E,p^\infty}\colon \GQ \to \Aut(T_p(E)).$$
If we choose a $\ZZ_p$-basis of $T_p(E)$, then we may consider $\rho_{E,p^\infty}\colon \GQ\to \Aut(T_p(E))\simeq \GL(2,\ZZ_p)$, and we are interested in describing the image of $\rho_{E,p^\infty}$ in $\GL(2,\ZZ_p)$. Much is known about the image of $\rho_{E,p}$, most notably Serre's so-called open image theorem.

\begin{thm}[Serre, \cite{serre1}]\label{thm-serre2}
Let $E/\QQ$ be an elliptic curve without complex multiplication and, for each prime $p$, let $G_p\subseteq \GL(2,\ZZ_p)$ be the image of $\rho_{E,p^\infty}$. Then, $G_p$ is an open subgroup of $\GL(2,\ZZ_p)$ for every $p$ (in particular, the index is finite), and $G_p=\GL(2,\ZZ_p)$ for all but finitely many primes.
\end{thm}

In a recent article \cite{zywina2}, Zywina has determined (up to a finite number of $j$-invariants) a finite list of all possible indices that may occur for the image of the representation $\rho_{E}\colon \GQ \to \GL(2,\widehat{\ZZ})$ that results as inverse image of $\rho_{E,n}\colon \GQ \to \Aut(E[n])\simeq \GL(2,\ZZ/n\ZZ)$.

Rouse and Zureick-Brown have classified all the possible $2$-adic images of $\rho_{E,2^\infty}\colon \GQ\to \GL(2,\ZZ_2)$, and Sutherland and Zywina have conjectured the possibilities for the mod $p$ image for all primes $p$.

\begin{thm}[Rouse, Zureick-Brown, \cite{RZB}]\label{thm-rzb} Let $E$ be an elliptic curve over $\QQ$ without complex multiplication. Then, there are exactly $1208$ possibilities for the $2$-adic image $\rho_{E,2^\infty}(\GQ)$, up to conjugacy in $\GL(2,\ZZ_2)$. Moreover, the index of $\rho_{E,2^\infty}(\Gal(\overline{\QQ}/\QQ))$ in $\GL(2,\ZZ_2)$ divides $64$ or $96$, and every image is defined at most modulo $32$.
\end{thm}

\begin{conj}[Sutherland, Zywina, \cite{zywina1}]\label{conj-zyw} Let $E/\QQ$ be an elliptic curve. Let $G\subseteq \GL(2,\ZZ/p\ZZ)$ be the image of $\rho_{E,p}$. Then, there are precisely $63$ isomorphism types of images. 
\end{conj}

Let us now include here some elementary results that we will use in our proofs in the next section. First, the existence of the Weil pairing implies that the roots of unity $\QQ(\zeta_n)$ are contained in the $n$-th division field.

\begin{prop}\label{prop-cyclotomic}
Let $E/\QQ$ be an elliptic curve, let $n$ be a positive integer. Then, $\det(\rho_{E,n})=\chi_n$ is the $n$-th cyclotomic character. In particular, if we let $\zeta_n$ be any primitive $n$-th root of unity, then $\QQ(\zeta_n)\subseteq \QQ(E[n])$, and for any $\sigma\in\GQ$ we have $\sigma(\zeta_{n})=(\zeta_{n})^{\det(\rho_{E,n}(\sigma))}$.
\end{prop}
\begin{proof}
See \cite[Chapter III, Corollary 8.1.1]{silverman} and \cite[Lecture 6 Theorem 6.3]{ConradRubin}.
\end{proof}

\begin{corollary}\label{cor-order_p(p-1)k}\label{cor-det_gens}
Let $E/\QQ$ be an elliptic curve, let $p>2$ a prime, let $m,n\geq 1$, and suppose that $\QQ(\zeta_{p^n})\subseteq \QQ(E[m])$. Let $\sigma\in \GQ$ be such that its restriction to $\QQ(\zeta_{p^n})$ generates the cyclic group $\Gal(\QQ(\zeta_{p^n})/\QQ)$. Then, the image of $\sigma$ through $\rho_{E,m}$ is an element of order divisible by $\varphi(p^n)=p^{n-1}(p-1)$.
\end{corollary}

\begin{proof}
This follows immediately from Proposition \ref{prop-cyclotomic}, and  the fact that if  $\QQ(\zeta_{p^n})\subseteq \QQ(E[m])$, then the restriction map of Galois groups  $\Gal(\QQ(E[m])/\QQ)\to\Gal(\QQ(\zeta_{p^n})/\QQ)\simeq(\ZZ/p^n\ZZ)^\times$ is surjective, and $(\ZZ/p^n\ZZ)^\times$ is cyclic because $p>2$.\end{proof}

Central to many of our arguments will be the fact that $\QQ(\zeta_n) \subseteq \QQ(E[n])$, but in order to prove Theorem \ref{thm:main_horizontal} we will need to better understand the fields of the form $K_E(p) = \QQ(E[p]) \cap \QQ^{ab}$. In Section \ref{sec:abelian_exts} we classify just how large the fields $K_E(p)$ can be. In order to do this we break the problem down into cases depending on what maximal group the image of $\rho_{E,p}$ is contained in. 

\begin{prop}[\cite{serre1}]\label{prop:max_groups} Let $E/\Q$ be an elliptic curve and let $p$ be a prime. Let $G$ be the image of $\rho_{E,p}\colon \GQ\to \Aut(E[p])\simeq \GL(2,\ZZ/p\ZZ)$. Then, there is exists a $\ZZ/p\ZZ$-basis for $E[p]$ such that on of the following is true:
\begin{enumerate}
\item $G = \GL(2,\ZZ/p\ZZ)$;
\item $G$ is contained in a Borel subgroup of $\GL(2,\ZZ/p\ZZ)$;
\item $G$ is contained in the normalizer of a split Cartan subgroup of $\GL(2,\ZZ/p\ZZ)$;
\item $G$ is contained in the normalizer of a non-split Cartan subgroup of $\GL(2,\ZZ/p\ZZ)$;
\item $G$ is contained in one of a finite list of ``exceptional'' subgroups.
\end{enumerate}
\end{prop}

Question \ref{ques:main} is best phrased and studied within the context of Galois representations. Fix an integer $n\geq 2$, and fix a $\ZZ/n\ZZ$-basis of $E[n]$. The absolute Galois group of $\QQ$ acts on $E[n]$ and induces a Galois representation $\rho_{E,n}\colon \GQ\to \Aut(E[n])\simeq\GL(2,\ZZ/n\ZZ)$, such that $\Ker(\rho_{E,n})=\Gal(\overline{\QQ}/\QQ(E[n]))$. Let us denote $\kappa_n = \Ker(\rho_{E,n})$ and $G_n=\Im(\rho_{E,n})\subseteq \GL(2,\ZZ/n\ZZ)$. Then, part (1) of Question \ref{ques:main} asks when is it possible that $\kappa_m=\kappa_n$ for distinct integers $m,n\geq 2$. Theorem \ref{thm:main_vertical} studies elliptic curves with $\kappa_{p^n}=\kappa_{p^{n+1}}$ or, equivalently, curves $E$ such that $G_{p^{n+1}}$ is isomorphic to $G_{p^n}$. If we denote the reduction mod $p^n$ map by $\pi_{p^n}\colon \GL(2,\ZZ/p^{n+1}\ZZ)\to \GL(2,\ZZ/p^n\ZZ)$ and $Z_{p^n}=\Ker(\pi_{p^n})$, then we are trying to find elliptic curves such that $G_{p^{n+1}}\cap Z_{p^n}$ is trivial. 

\begin{example}
	Let us look at Example \ref{ex-2} from the point of view of Galois representations. Let $E\colon y^2=x^3+13x-34$ and let $E(\QQ(i))_\text{tors}=E[4]=\langle P,Q\rangle$ where $P=(7,-20)$ and $Q=(-3-10i,30+10i)$. Then, 
	$\overline{Q}=(-3+10i,30-10i) = P+3Q$. In particular, if we write $\rho_{E,4}$ with respect to the basis $\{P,Q\}$ of $E[4]$, then the image $G_4$ is 
	$$\left\{ \left(\begin{matrix}
	1 & 0\\
	0 & 1
	\end{matrix}\right) , \left(\begin{matrix}
	1 & 1\\
	0 & 3
	\end{matrix}\right) \right\}\subseteq \GL(2,\ZZ/4\ZZ),$$
	while $Z_2 = \operatorname{Id}+2\cdot M(2,\ZZ/4\ZZ)$, where $M(2,\ZZ/4\ZZ)$ are the $2\times 2$ matrices with coordinates in $\ZZ/4\ZZ$. Hence, $Z_2\cap G_4$ is trivial, as claimed, and $\Gal(\QQ(E[4])/\QQ)\simeq G_4\simeq G_2\simeq \Gal(\QQ(E[2])/\QQ)$. Since $\QQ(E[2])\subseteq \QQ(E[4])$ we conclude that the $2$-nd and $4$-th division fields are actually equal.
\end{example}

From the point of view of representations and kernels of reduction maps, our Theorem \ref{thm:main_vertical} is at the opposite side of the spectrum from the following theorem of Dokchitser, Dokchitser, and Elkies, which determines when $Z_{p^n}\subseteq G_{p^{n+1}}$. 
\begin{thm}[Serre \cite{serre1}, Elkies \cite{elkies}, Dokchitser, Dokchitser \cite{dok}]\label{thm-dde} Let $E/\QQ$ be an elliptic curve, let $p$ be a prime, and let $n\geq 1$. If $\rho_{E,p^n}$ is surjective, then $\rho_{E,p^{n+1}}$ is surjective, unless $p^n=2$, $3$, or $4$. Moreover, the $j$-invariants of elliptic curves where $\rho_{E,p^n}$ is surjective but $\rho_{E,p^{n+1}}$ is not, are given explicitly by $1$-parameter families.
\end{thm}
Indeed, for $n\geq 1$ we have $G_{p^n}=\GL(2,\ZZ/p^n\ZZ)$ and $Z_{p^n}\subseteq G_{p^{n+1}}$if and only if  $G_{p^{n+1}}=\GL(2,\ZZ/p^{n+1}\ZZ)$, because the reduction map $\pi_{p^n}$ is surjective. Thus, Theorem \ref{thm-dde} shows that if $\pi_{p^n}(G_{p^{n+1}})=G_{p^n}$ and $Z_{p^n}\cap G_{p^{n+1}}\neq Z_{p^n}$, then $p^n=2$, $3$, or $4$. The fact that the surjectivity of $\rho_{E,p^n}$ implies $\rho_{E,p^{n+1}}$ for $p\geq 5$ and $n\geq 1$  was known by work of Serre (cf. \cite[IV-23, Lemma 3]{serre}).

\section{Coincidences in towers}\label{sec-towers}

The goal of this section is to prove Theorem \ref{thm:main_vertical}. In other words, this section is concerned with the possibility of an equality $\Q(E[p^n])=\Q(E[p^{n+1}])$ for some prime $p$ and $n\geq 1$. In the spirit of Question \ref{ques:main}, we are also interested in whether $\Q(E[p^n])\cap \Q(\zeta_{p^{n+1}})$ can be larger than $\Q(\zeta_{p^n})$. We answer these questions first for odd primes, and then we shall turn our attention to the case of $p=2$.

\subsection{The Case $p\geq 3$}\label{sec:odd_vert}

The goal of this section is to prove the case of Theorem \ref{thm:main_vertical} when $p$ is an odd prime. In fact, we would like to prove that $\Q(E[p^n])\cap \Q(\zeta_{p^{n+1}})=\Q(\zeta_{p^n})$. 

\begin{prop}\label{prop-pnroots}
Let $E/\QQ$ an elliptic curve and $p\geq 3$ be a prime. Then, for every $n\in\NN$, the field $\QQ(E[p^n])$ does not contain the $p^{n+1}$-th roots of unity. 
\end{prop}

As a corollary, we obtain:

\begin{thm}\label{prop-p_geq_3}
Let $E/\QQ$ be an elliptic curve and $p\geq 3$ a prime. Then, for every $n\in\NN$, we have that $\QQ(E[p^n])\neq \QQ(E[p^{n+1}])$.
\end{thm}
\begin{proof}
Suppose that $\QQ(E[p^n])=\QQ(E[p^{n+1}])$ for some $p\geq 3$ and $n\geq 1$. Then, $\QQ(\mu_{p^{n+1}})\subseteq \QQ(E[p^{n+1}])=\QQ(E[p^n])$ by Prop. \ref{prop-cyclotomic}, which contradicts the conclusion of Prop. \ref{prop-pnroots}.
\end{proof}

In order to show Proposition \ref{prop-pnroots}, we will need two lemmas about the elements of $\GL(2,\ZZ/p^n\ZZ)$ with $p$-power order and a lemma about element with order divisible by $\varphi(p^n) $.

\begin{lemma}\label{lem-max_prime_powers}
    If $p$ is a prime and $A\in\GL(2,\ZZ/p^n\ZZ)$ such that the order of $A$ is $p^k$ for some $k\in\NN$, then $1\leq k\leq n$ and all such orders occur.
\end{lemma}

\begin{proof}
We start by noticing that the matrix $\left(\begin{smallmatrix}
    1&p^{n-k}\\0&1
\end{smallmatrix}
  \right)$ has order $p^k$ and so $\GL(2,\ZZ/p^n\ZZ)$ has an element of order $p^k$ for every $1\leq k \leq n$. To prove that there are no elements of higher order, we proceed by induction on $n$. The case of $n=1$ follows from Lagrange's theorem and the fact that $\#\GL(2,\ZZ/p\ZZ)=p(p-1)(p^2-1).$

Next, we assume that the lemma is true for $n=\ell$ and let $A\in\GL(2,\ZZ/p^{\ell+1}\ZZ)$ such that $A$ has order $p^k$ for some $k\in\NN$. We aim to show that $A^{p^{\ell+1}}$ is the identity and so the order of $A$ divides $p^{\ell+1}$. Since $A^{p^k}=\operatorname{Id}$, it follows that $A^{p^k}\equiv \operatorname{Id}\bmod p^\ell$ and, by the induction hypothesis, the order of $A\bmod p^\ell$ must be $p^{k'}$ with $1\leq k'\leq \ell$. Thus, we have $A^{p^\ell} \equiv  \operatorname{Id} \bmod p^\ell$, and so we may write  $A^{p^\ell} = \left(\begin{smallmatrix}
1+ap^\ell&bp^\ell\\
cp^\ell&1+dp^\ell
\end{smallmatrix}\right)$ for some $a,b,c,d\in\{0,1,\dots,p-1\}$. By induction on $m$ one can show that 
$$(A^{p^\ell})^m = \left(\begin{matrix}
1+map^\ell & mbp^\ell\\
mcp^\ell & 1+mdp^\ell
\end{matrix}\right)\in \GL(2,\ZZ/p^{\ell+1}\ZZ)$$
and so $A^{p^{\ell+1}} \equiv  (A^{p^\ell})^p \equiv  \left(\begin{smallmatrix}1&0\\0&1 \end{smallmatrix}\right) \bmod p^{\ell+1}$. Thus, the order of $A$ divides $p^{\ell+1}$, which completes the induction step, and we conclude the proof.
\end{proof}

\begin{lemma}\label{lem-ker_orders}
    Let $p$ be a prime, $n\in\NN$, and let $\pi_{p,n}\colon  \GL(2,\ZZ/p^n\ZZ)\to \GL(2,\ZZ/p\ZZ)$ be the natural reduction modulo $p$ map. Then, the elements in $\ker\pi_{p,n}$ have order dividing $p^{n-1}$. 
\end{lemma}

\begin{proof} Again, we proceed by induction on $n$. The case of $n=1$ follows from the fact that $\pi_{p,1}$ is the identity map and the only element in the kernel is the identity which has order $1 = p^{0} = p^{n-1}$.

Next, assume that the result is true for some $k\in \NN$. Let $A\in\ker\pi_{p,k+1}$ and let $\bar{A}$ be the image of $A$ under the reduction map $\GL(2,\ZZ/p^{k+1}\ZZ)\to\GL(2,\ZZ/p^k\ZZ)$. Then, by assumption, we know that $\bar{A}^{p^{k-1}} = \operatorname{Id}\in\GL(2,\ZZ/p^k\ZZ)$, thus
$$A^{p^{k-1}}= \begin{pmatrix}
1+p^ka& p^kb\\p^kc&1+p^kd
\end{pmatrix},$$
for some $a,b,c,d\in\ZZ/p^k\ZZ$.
Again by induction on $m$,
$$\left(A^{p^{k-1}}\right)^m= \begin{pmatrix}
1+mp^ka& mp^kb\\mp^kc&1+mp^kd
\end{pmatrix},$$
and, in particular, $A^{p^{k}} = \left(A^{p^{k-1}}\right)^p = \operatorname{Id}\in\GL(2,\ZZ/p^{k+1}\ZZ).$ This finishes the proof.
\end{proof}

\begin{lemma}\label{lem:det_is_sq}
If $A_n\in \GL(2,\ZZ/p^n\ZZ)$ is a matrix with order divisible by $\varphi(p^{n+1})=p^n(p-1)$, then $\det(A_n)$ is a square modulo $p$. 
\end{lemma}

\begin{proof}
We start by writing $\operatorname{ord}(A_n)=p^{n}(p-1)k$, for some $k\geq 1$. 
From Lemma \ref{lem-max_prime_powers}, we know that are no elements of order $p^{n+1}$ in $\GL(2,\ZZ/p^n\ZZ)$ and so $p$ cannot divide $k$. 
Thus $\gcd(p^n,(p-1)k)=1$ and there exist integers $x$ and $y$ such that $1=x(p-1)k + yp^n$. Letting $B_n=A_n^{x(p-1)k}$ and $C_n=A_n^{yp^n}$, we have that $B_n$ and $C_n$ have order $p^n$ and $(p-1)k$, respectively, and $A_n=A_n^{x(p-1)k+yp^n}= A_n^{x(p-1)k}A_n^{yp^n}=B_nC_n$. 
Moreover $A_n = B_n C_n = C_n B_n$, since $B_n$ and $C_n$ are both powers of $A_n$.

Next, let $\pi_{p,n}$ be as in the statement of Lemma \ref{lem-ker_orders}, $B_1 = \pi_{p,n}(B_n)$ and $C_1 = \pi_{p,n}(C_n)$. Since the $\# \GL(2,\ZZ/p\ZZ) = p(p-1)^2(p+1)$ and the order of $B_n$ is $p^n$, we know that $\operatorname{ord}(B_1)$ divides $p$. However, any element in $\GL(2,\ZZ/p\ZZ)$ of order dividing $p$ is conjugate to a matrix in the subgroup 
$$H = \left\{\begin{pmatrix}
    1&x\\0&1
\end{pmatrix}: x \in \ZZ/p\ZZ\right\}.$$
Since conjugation doesn't change the determinant of $A_n$, we can assume without loss of generality that $B_1$ is in $H$. Since $B_n$ has order $p^n$, Lemma \ref{lem-ker_orders} gives that $B_n$ does \emph{not} belong to $\Ker(\pi_{p,n})$, and therefore $B_1$ is \emph{not} the identity element of $\GL(2,\ZZ/p\ZZ)$.

Thus, $B_1 = \left(\begin{smallmatrix}
1&a\\0&1     
 \end{smallmatrix}\right)$ for some fixed $a\in(\ZZ/p\ZZ)^\times$ and since $B_nC_n = C_n B_n$ and $\pi_n$ is a group homomorphism, we know that $B_1C_1 = C_1B_1$. Setting $C_1 = \left(\begin{smallmatrix}
     \alpha&\beta\\\gamma&\delta
 \end{smallmatrix} \right)$ and using that $B_1$ and $C_1$ commute we have that
\begin{align*}
    \alpha &\equiv \alpha+\gamma a \bmod p,\\
\beta+\alpha a &\equiv \beta+\delta a \bmod p.    
\end{align*}
Since $a\not\equiv0 \bmod p$, it must be that $\gamma \equiv 0 $ and $\alpha \equiv \delta\bmod p$. Thus, $C_1 = \left(\begin{smallmatrix}
    \alpha&\beta \\ 0 &\alpha
\end{smallmatrix} \right)$. Finally, we have that $\det(A_1)=\det(B_1)\det(C_1)\equiv \alpha^2$ and $\det(A_n)\equiv \det(A_1)\equiv \alpha^2 \bmod p$, as desired. 
\end{proof}

We are now ready to prove Proposition \ref{prop-pnroots}.

\begin{proof}[Proof of Prop. \ref{prop-pnroots}]
    Suppose towards a contradiction that there is an elliptic curve $E/\QQ$, a prime $p\geq 3$, and $n\in\NN$ such that $\QQ(\mu_{p^{n+1}})\subseteq \QQ(E[p^n])$. 
    Let $\zeta_{p^{n+1}}$ be a fixed primitive $p^{n+1}$-th root of unity, $\zeta_{p^n} = \zeta_{p^{n+1}}^p$ be the corresponding primitive $p^n$-th root of unity and let $\sigma\in G_\QQ$ be any element such that $\sigma\big|_{\QQ(\zeta_{p^{n+1}})}$ generates $\Gal(\QQ(\zeta_{p^{n+1}})/\QQ)$. Further, let $A_{n} = \rho_{E,p^{n}}(\sigma)$.

 By Corollary \ref{cor-order_p(p-1)k}, if $\QQ(\zeta_{p^{n+1}})\subseteq \QQ(E[p^n])$, then the order of $A_n=\rho_{E,p^n}(\sigma)\in \GL(2,\ZZ/p^n\ZZ)$ is divisible by $\varphi(p^{n+1})$. However, recall that we assumed that $\sigma$ restricted to $\QQ(\zeta_{p^{n+1}})$ generates the full group Galois group of $\QQ(\zeta_{p^{n+1}})/\QQ$ and so $\sigma$ restricted to $\QQ(\zeta_{p^n})$ generates $\Gal(\QQ(\zeta_{p^n})/\QQ)$. Therefore, $\det(\rho_{E,p^n}(\sigma))=\det(A_n)$ must be a generator of $(\ZZ/p^n\ZZ)^\times$. But from Lemma \ref{lem:det_is_sq}, we know that $\det(A_n)$ is a square mod $p$. Thus, we have reached a contradiction as a quadratic residue cannot be a generator for $p\geq 3$, and the proof is complete. 
\end{proof}

%Since $\QQ(\zeta_{p^{n+1}})\subseteq \QQ(E[p^{n+1}]) = \QQ(E[p^n])$, there must be a $\sigma\in G_\QQ$ such that $\sigma$ restricted to $\QQ(\zeta_{p^{n+1}})$ generates $\Gal(\QQ(\zeta_{p^{n+1}})/\QQ)$ and $\rho_{E,p^n}(\sigma)$ has order $(p-1)p^n$. Let $A = \rho_{E,p^n}\in \GL(2,\ZZ/p^n\ZZ)$. From basic group theory, we know that we find $B_1,B_2\in\GL(2,\ZZ/p^n\ZZ)$ such that $A = B_1 B_2 = B_2 B_1$ with $o(B_1) = p^n$ and $o(B_2) = p-1$.

\subsection{The case $p=2$, without CM}\label{sec:even_vert}

In this subsection we continue the proof of Theorem \ref{thm:main_vertical} by studying when $\Q(E[2^{n+1}])=\Q(E[2^n])$ for some $n\geq 1$. We shall consider two cases, according to whether $E/\Q$ has complex multiplication. In this section, we consider the non-CM case.

In the non-CM case, the work of Rouse and Zureick-Brown (Theorem \ref{thm-rzb}) reduces the  proof of Theorem \ref{thm:main_vertical} in the case of $p=2$ to a finite computation. Indeed, by Theorem \ref{thm-rzb}, if $E/\QQ$ is an elliptic curve with no CM, there are $1208$ possible images for $\rho_{E,2^\infty}$, each one defined at most modulo $32$ (i.e., $\Im \rho_{E,2^\infty}$ is always the full inverse image of $\Im \rho_{E,2^5}$ under the reduction map $\GL(2,\ZZ_2)\to \GL(2,\ZZ/32\ZZ)$). An immediate consequence of this fact is that if $n\geq 5$, then we cannot have $\Im {\rho_{E,2^{n+1}}}\simeq \Im \rho_{E,2^n}$, and therefore $\Gal(\QQ(E[2^{n+1}])/\QQ)$ is not isomorphic to $\Gal(\QQ(E[2^n])/\QQ)$. In particular, if $\QQ(E[2^{n+1}])=\QQ(E[2^n])$ we must have $n<5$. The database \cite{RZB} provides generators of a subgroup $G_5\subseteq \GL(2,\ZZ/32\ZZ)$ for each of the $1208$ possible images. If we let $G_k$ be the image of $G_5$ in $\GL(2,\ZZ/2^k\ZZ)$, for $1\leq k \leq 5$, then we have carried out an exhaustive search of the possible $2$-adic images for examples where $G_{n+1}\simeq G_n$, for some $1\leq n\leq 4$, so that $\QQ(E[2^{n+1}])=\QQ(E[2^n])$. There are precisely two types of images with this property, namely $\tt X_{20b}$ and $\tt X_{60d}$ (in the notation of \cite{RZB}) and in both cases we had $G_2\simeq G_1$, i.e., $\QQ(E[4])=\QQ(E[2])$. They differ, however, in the fact that $G_2\simeq \ZZ/2\ZZ$ for $\tt X_{20b}$ while $G_2\simeq S_3$ for $\tt X_{60d}$. Our search, thus, yields the following result.

\begin{prop}\label{prop:even_vert}
Let $E/\QQ$ be an elliptic curve without CM,  such that $\QQ(E[2^{n+1}]) = \QQ(E[2^n])$ for some $n\geq 1$. Then $n=1$, and there is a $t\in\QQ$ such that $E$ is $\QQ$-isomorphic over $\QQ$ to an elliptic curve of the form $E'\colon y^2 = x^3 + A(t)x + B(t)$, where 
\begin{align*}
A(t) &= -27t^8+648t^7-4212t^6-2376t^5+60102t^4+79704t^3-105732t^2-235224t-107811,\\
B(t) &= 54t^{12}-1944t^{11}+24300t^{10}-97848t^9-251262t^8+1722384t^7+4821768t^6\\&\ \ \ -8697456t^5-64323558t^4-140447736t^3-157012020t^2-90561240t-21346578.
\end{align*}
\end{prop}

\begin{remark}\label{rem-groupG}
We point out here that, in fact, the family in Proposition \ref{prop:even_vert} contains the family in Theorem \ref{thm:abelian}. The family in Theorem \ref{thm:abelian} corresponds to elliptic curves with $\QQ(E[2]) = \QQ(E[4]) = \QQ(i)$, while the family above corresponds to elliptic curves with $\QQ(E[2]) = \QQ(E[4])$ an $S_3$ extension of $\QQ$, with $\QQ(i)$ the unique quadratic subfield. These two curves correspond to $\tt X_{20b}$ and $\tt X_{60d}$, in the notation of \cite{RZB}, and the map between these curves can be computed from the information there. 

The curves in Proposition \ref{prop:even_vert} all have $\Im\rho_{E,4}$ conjugate to a subgroup of
$$G = \left\langle\begin{pmatrix} 1&0\\ 3&3 \end{pmatrix}, \begin{pmatrix} 3&3\\ 1&0 \end{pmatrix}  \right\rangle\subseteq \GL(2,\Z/4\Z),$$
while the curves in Theorem \ref{thm:abelian} have $\Im\rho_{E,4}$ conjugate to\footnote{The image must actually be \emph{equal} to $H$ since $\#H = 2$ and $\QQ(i)\subseteq \QQ(E[4])$.} 
$$H =  \left\langle\begin{pmatrix} 1&1\\ 0&3 \end{pmatrix}  \right\rangle.$$
Since $H$ is conjugate to a subgroup of $G$, curves with images in $H$ arise as points on the modular curve $X_G$.

\end{remark}

\begin{example}
   Let $E/\QQ$ be the elliptic curve with Cremona label \href{https://www.lmfdb.org/EllipticCurve/Q/162d1/}{\texttt{162d1}} which is given by Weierstrass equation $y^2 + xy + y = x^3 - x^2 + 4x - 1$. Using Magma \cite{Magma} we see that, $E(\QQ)_{\tor}$ is trivial and that $\QQ(E[2]) = \QQ(E[4]) = \QQ(\alpha)$ where $\alpha$ is a root of $f(x) = 256x^6 + 6624x^4 + 42849x^2 + 82944$. 
\end{example}

%\section{Conclusion}

\subsection{The case $p=2$, with CM}\label{sec:even_vert_CM}

In this subsection we complete the proof of Theorem \ref{thm:main_vertical} by studying when $\Q(E[2^{n+1}])=\Q(E[2^n])$ for some $n\geq 1$, and $E/\Q$ has complex multiplication. 

\begin{prop}\label{prop-CMvert}
	Let $E/\Q$ be an elliptic curve with complex multiplication, and let $n\geq 1$. Then $\Q(E[2^n])\subsetneq \Q(E[2^{n+1}])$.
\end{prop}
\begin{proof} 
Suppose first that $E/\Q$ is an elliptic curve with complex multiplication, and $\Q(E[2])=\Q(E[4])$. Then, the image $G_4$ of $\rho_{E,4}$ is a group such that its reduction $G_2$ modulo $2$ satisfies $\#G_2 = \# G_4$. A Magma computation shows that any such group is a conjugate of a subgroup of 
$$G = \left\langle\begin{pmatrix} 1&0\\ 3&3 \end{pmatrix}, \begin{pmatrix} 3&3\\ 1&0 \end{pmatrix}  \right\rangle\subseteq \GL(2,\Z/4\Z),$$ 
the group that already appeared in Remark \ref{rem-groupG}. The modular curve $X_G$ corresponds to $X_{20b}$ (genus $0$) in the notation of Rouse and Zureick-Brown, and they have computed the $j$-line $j\colon X_G \to \mathbb{P}^1$, which is given as follows: 
$$j_G(t) = \frac{-4t^8 + 32t^7 + 80t^6 - 288t^5 - 504t^4 + 864t^3 + 1296t^2 - 864t - 1188}{t^4 + 4t^3 + 6t^2 + 4t + 1}.$$
In order to rule out elliptic curves with CM that have an image of $\rho_{E,4}$ that is a conjugate of a subgroup of $G$, it suffices to show that if $j_0$ is a rational CM $j$-invariant, then $j_0$ is not a value of $j_G(t)$ for some rational value of $t$. There are precisely $13$ such rational CM $j$-invariants, namely
$0$, $54000$, $-12288000$, $1728$, $287496$, $-3375$, $16581375$, $8000$, $-32768$, $-884736$,
$-884736000$, $-147197952000$, $-262537412640768000$, and one can check, one by one (again using Magma, for example) that $j_G(t)=j_0$ is impossible for rational values of $t$. Hence, $\Q(E[2])=\Q(E[4])$ is impossible for elliptic curves over $\Q$ with CM. 

It remains to show that $\Q(E[2^n])=\Q(E[2^{n+1}])$ is impossible for $n>1$ in the CM case. Suppose, for a contradiction, that $E/\Q$ is an elliptic curve with CM by an imaginary quadratic field $K$ and suppose that $\Q(E[2^n])=\Q(E[2^{n+1}])$ for some $n>1$. In particular, since $j(E)\in\Q$, we have that $K(j(E))=K$, and $K(E[2^n])=K(E[2^{n+1}])$. 
In this proof, we will argue in terms of the following diagram:
$$
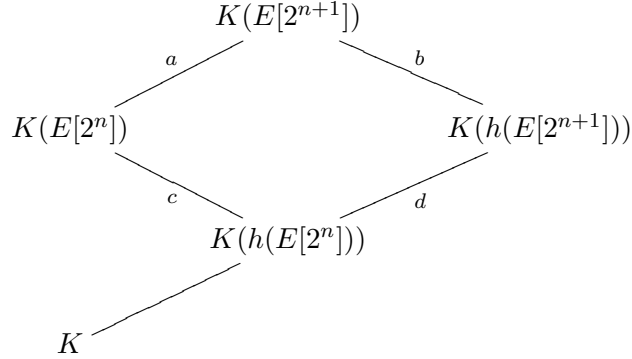
\xymatrix{
	& K(E[2^{n+1}]) & \\
	K(E[2^n]) \ar@{-}[ur]^a&  & K(h(E[2^{n+1}])) \ar@{-}[ul]_b \\
	& K(h(E[2^n])) \ar@{-}[ur]_d \ar@{-}[ul]^c &  \\
	K \ar@{-}[ur] & & }$$
where $h$ is a Weber function for $E$ (see \cite{lozano1}, Definition 2.4). Our initial assumption is that $a=1$. Theorem 4.4 of \cite{lozano1} shows that, for any $n\geq 2$, we have $d=[H_f(h(E[2^{n+1}])):H_f(h(E[2^n]))] =2^2$, where $H_f=K(j(E))=K$ in this case. Hence, $c=[K(E[2^n]):K(h(E[2^n]))]$ must be divisible by $4$. Theorem 4.1 of \cite{lozano1} shows that $c$ is a divisor of $\#\mathcal{O}_{K,f}^\times$, where $E$ has CM by the order $\mathcal{O}_{K,f}$ of $K$. Since $\# \mathcal{O}_{K,f}^\times$ is always a divisor of $4$ or $6$, we must have $\# \mathcal{O}_{K,f}^\times =4$ and therefore $\mathcal{O}_{K,f}=\Z[i]$ and $K=\Q(i)$, and $j(E)=1728$. The finite list of possible $2$-adic images for elliptic curves over $\Q$ with CM and $j=1728$ are described in Theorem 1.7 of \cite{lozano1}, and one can verify that, in all cases, $[\Q(E[8]):\Q(E[4])]=2$ or $4$, and $[\Q(E[2^{n+1}]):\Q(E[2^n])]=4$ for all $n\geq 3$. Hence, this shows that $\Q(E[2^n])\neq \Q(E[2^{n+1}])$ for any $n\geq 1$ and any CM elliptic curve over $\Q$.
\end{proof} 

Together, Propositions \ref{prop-p_geq_3}, \ref{prop:even_vert}, and \ref{prop-CMvert} imply Theorem \ref{thm:main_vertical}.

We finish this section discussing the possibility of $2^{n+1}$-th roots of unity in the $2^n$-th division field of an elliptic curve. For example, the elliptic curve $E: y^2=x^3+x$ satisfies $\Q(E[2])=\Q(\zeta_4)$, and the curve $E: y^2=x^3+2x$ satisfies $\Q(\zeta_8)\subseteq \Q(E[4])$. Next, we show that the elliptic curve $y^2=x^3-11x-14$,  with Cremona label \href{https://www.lmfdb.org/EllipticCurve/Q/32a3/}{\texttt{32a3}}, satisfies that the $2^n$-th division field contains the $2^{n+1}$-roots of unity.
	
\begin{theorem}\label{thm-2nroots}
		Let $E/\Q$ be the elliptic curve $y^2=x^3-11x-14$ (i.e. the curve with Cremona label \href{https://www.lmfdb.org/EllipticCurve/Q/32a3/}{\texttt{32a3}}). Then, $\Q(\zeta_{2^{n+1}})\subseteq \Q(E[2^n])$ for all $n>1$.
\end{theorem}
\begin{proof} Let $E/\Q$ be the elliptic curve $y^2=x^3-11x-14$. In \cite{lozano1}, Example 9.4, it is shown that the $2$-adic image of $E/\Q$ is given by 
	$$G=\left\langle A=\begin{pmatrix} -1&0\\ 0&1 \end{pmatrix}, B=\begin{pmatrix} 5&0\\ 0&5 \end{pmatrix},C=\begin{pmatrix} -1&-1\\ 4&-1 \end{pmatrix} \right\rangle \subseteq \GL(2,\Z_2).$$
	Note that the generators of $G$ are subject to the following relations:
	$$BC=CB,\ AB=BA,\ ACA=BC^{-1}.$$
	Further, $B^{2^{n-2}}\equiv C^{2^n}\equiv \operatorname{Id} \bmod 2^n$. It follows that $D=ACAC^{-1}=BC^{-2}$ satisfies
	$$D^{2^{n-1}} \equiv B^{2^{n-1}}C^{-2^n}\equiv 1 \bmod 2^n.$$
	Moreover, $G/\langle D \rangle $ is abelian. It follows that $G/\langle D \rangle \bmod 2^n$ is abelian, of size 
	$$\frac{|G \bmod 2^n|}{|\langle D \rangle \bmod 2^n|} = \frac{2\cdot 2^{n-2}\cdot 2^n}{2^{n-1}} = 2^n.$$
	Moreover, $D^{2^{n-2}}\equiv B^{2^{n-2}}C^{-2^{n-1}}\equiv C^{-2^{n-1}} \bmod 2^n$. Thus, $C^{2^{n-1}}\equiv \operatorname{Id} \bmod (2^n,\langle D\rangle)$. It follows that $\langle A,C\rangle/\langle D\rangle$ is of size $2^n$, and therefore $A$ and $C$ span all of $G/\langle D\rangle$. Hence, $G/\langle D\rangle \simeq \Z/2\Z\times \Z/2^{n-1}\Z$. We conclude that there is an abelian extension $F_n/\Q$, with $F_n\subseteq \Q(E[2^n])$, such that its Galois group is given by  $\Gal(F_n/\Q)\simeq  \Z/2\Z\times \Z/2^{n-1}\Z$. 
	
	Now, notice that the conductor of $E/\Q$ is $32$, and therefore, by the criterion of N\'eron--Ogg--Shafarevich, the extension $\Q(E[2^n])/\Q$ (and therefore $F_n/\Q$) is only ramified above $2$. In particular, $F_n/\Q$ is an abelian extension of $\Q$ that is only ramified above $2$, and we conclude that $F_n\subseteq \Q(\zeta_{2^\infty})$. Further, $\Gal(\Q(\zeta_{2^\infty})/\Q)\simeq \Z/2\Z \times \Z_2$, and $n>1$, and so there exists a unique extension $F_n\subseteq \Q(\zeta_{2^\infty})$ such that $\Gal(F_n/\Q)\simeq \Z/2\Z\times \Z/2^{n-1}\Z$, namely $F_n = \Q(\zeta_{2^{n+1}})$. 
	
	Hence, $F_n = \Q(\zeta_{2^{n+1}})\subseteq \Q(E[2^n])$ for $n>1$, as we claimed.
\end{proof}

\section{The Abelian Case}\label{sec-abelian}

In this section we prove Theorem \ref{thm:abelian}. We shall rely on the classification of abelian division fields given in \cite{LR+GJ}. We cite here the main result of that paper for reference:

\begin{thm}\label{thm:LR+GJ} {\rm \cite[Theorem 1.1]{LR+GJ}} Let $E/\Q$ be an elliptic curve. If there is an integer $n\geq 2$  such that $\Q(E[n])=\Q(\zeta_n)$, then $n=2,3,4,$ or $5$. More generally, if $\Q(E[n])/\Q$ is abelian, then  $n=2,3,4,5,6$, or $8$.  Moreover, $\Gal(\Q(E[n])/\Q)$ is isomorphic to one of the following groups:
	\begin{center}
		% % \begin{table}[h!]
		\renewcommand{\arraystretch}{1.2}
		\begin{tabular}{|c||c|c|c|c|c|c|}
			\hline
			$n$ & $2$ & $3$ & $4$ & $5$ & $6$ & $8$\\
			\hline 
			\multirow{4}{*}{$\Gal(\Q(E[n])/\Q)$}  & $\{0\}$ & $\Z/2\Z$ & $\Z/2\Z$ & $\Z/4\Z$ & $(\Z/2\Z)^2$ & $(\Z/2\Z)^4$ \\
			& $\Z/2\Z$ & $(\Z/2\Z)^2$ & $(\Z/2\Z)^2$ & $\Z/2\Z\times \Z/4\Z$ & $(\Z/2\Z)^3$ & $(\Z/2\Z)^5$ \\
			& $\Z/3\Z$ &  & $(\Z/2\Z)^3$ & $(\Z/4\Z)^2$ & & $(\Z/2\Z)^6$\\
			&  &  & $(\Z/2\Z)^4$ & & & \\
			\hline
		\end{tabular}
		% % \end{table} 
	\end{center}
	Furthermore, each possible Galois group occurs for infinitely many distinct $j$-invariants.
\end{thm} 

We are ready to present the proof of Theorem \ref{thm:abelian}.

\begin{proof}[Proof of Theorem \ref{thm:abelian}]
	Suppose first that $p$ and $q$ are distinct primes, with $\Q(E[p])$ abelian, and $\Q(E[p])\cap \Q(\zeta_{q^k})$ non-trivial, for some $k\geq 1$. By Theorem \ref{thm:LR+GJ}, we have $p=2,3$, or $5$. So we distinguish three cases depending on the value of $p$:
	\begin{enumerate}
		\item If $p=2$, then $\Q(E[2])$ is either trivial, quadratic, or cyclic cubic. For any prime $q$, the curve $y^2=x^3+(-1)^{(q-1)/2}q x$ shows that $\Q(E[2])\cap \Q(\zeta_{q^k})$ can be quadratic. Similarly, if $\Q(\zeta_{q^k})$ contains a cyclic cubic extension (i.e., if $3$ divides $(q-1)q$) given by a cubic polynomial $f(x)$, then the curve $y^2=f(x)$ shows that $\Q(E[2])\cap \Q(\zeta_{q^k})$ can be a cyclic cubic. 
		\item If $p=3$, the existence of the Weil pairing forces $\Q(\zeta_3)\subseteq \Q(E[3])$. Then, by Theorem \ref{thm:LR+GJ}, the field $\Q(E[3])$ is either quadratic, in which case $\Q(E[3])=\Q(\sqrt{-3})$, or is quartic with Galois group $(\Z/2\Z)^2$, so that $\Q(E[3])$ is the compositum of $\Q(\sqrt{-3})$ with another quadratic field $K$ over $\Q$. Thus, $\Q(E[3])\cap \Q(\zeta_{q^k})$ is at most quadratic.
		\item If $p=5$, then $\Q(\zeta_5)\subseteq \Q(E[5])$, and $\Gal(\Q(E[5])/\Q)\simeq \Z/4\Z$, or $\Z/2\Z\times \Z/4\Z$, or $(\Z/4\Z)^2$. Thus, $\Q(E[5])\cap \Q(\zeta_{q^k})$ is either quadratic of a cyclic quartic. As an example of the latter, consider $E: y^2 = x^3 - x^2 - 4319x + 100435$. Here $\Gal(\Q(E[5])/\Q)\simeq (\Z/4\Z)^2$ and one of the points of order $5$ is defined over a cyclic quartic field $F$ defined by $x^4+4x^2+2=0$. This extension has discriminant $2048$, and $F\subseteq \Q(\zeta_{16})$. Hence, $\Q(E[5])=F(\zeta_5)$, and $\Q(E[5])\cap \Q(\zeta_{16})=F$, a quartic field. 
	\end{enumerate}
	Thus, in all cases $\Q(E[p])\cap \Q(\zeta_{q^k})$ is trivial, quadratic, cyclic cubic, or cyclic quartic.
	
	It remains to consider when a full coincidence $\Q(E[n])=\Q(E[m])$ can occur. Suppose such a coincidence does occur. Then, in particular, $\Gal(\Q(E[n])/\Q)\simeq \Gal(\Q(E[m])/\Q)$.	Using the classification of abelian division fields in Theorem \ref{thm:LR+GJ}, it follows that $$(m,n) \in \{ (2,3), (2,4), (3,4), (3,6),(4,6),(4,8) \}.$$ By our results in Section \ref{sec-towers},  $\Q(E[4])=\Q(E[8])$ cannot happen, and the case of $\Q(E[2])=\Q(E[4])$ in the abelian case corresponds to the curve $\tt X_{60d}$ of \cite{RZB}, which is parametrized as in the statement of the theorem (see also Remark \ref{rem-groupG}). 
	
	The case of $(m,n) = (2,3)$ can be eliminated by seeing that in this case we would have $\Q(E[2])=\Q(E[3])=\Q(E[6])$. Thus, $\Gal(\Q(E[6])/\Q)\cong \Gal(\Q(E[3])/\Q)\cong \Gal(\Q(E[2]))$ but this is impossible according to the table in Theorem  \ref{thm:LR+GJ}. Similarly, if $\Q(E[3])=\Q(E[4])$ or if $\Q(E[4])=\Q(E[6])$ are abelian extensions of $\Q$, then $\Q(E[12])=\Q(E[3])\Q(E[4])=\Q(E[6])\Q(E[4])$ would be abelian, but this contradicts Theorem \ref{thm:LR+GJ}, which shows that $\Q(E[12])/\Q$ is never an abelian extension.
	
	Thus, it remains to consider the case of $\Q(E[3])=\Q(E[6])$. Since we have shown that $\Q(E[2])=\Q([3])$ cannot occur, we restrict our attention to the case of $\Q(E[2])\subsetneq \Q(E[3])=\Q(E[6])$. In particular, by Theorem \ref{thm:LR+GJ}, the extension $\Q(E[2])/\Q$ must be trivial or quadratic, and $\Q(E[6])=\Q(E[3])/\Q$ must be biquadratic. By the results of Section 6.2 of \cite{LR+GJ}, if $\Q(E[3])$ is abelian, then the image of $\rho_{E,3}$ is contained in a split Cartan subgroup of $\GL(2,\Z/3\Z)$ and, in particular, $E/\Q$ has two independent $3$-isogenies. Moreover, if $\Q(E[2])/\Q$ is trivial or quadratic, then $E/\Q$ has at least one $2$-torsion point defined over $\Q$. Notice that if $\Q(E[2])=\Q$, then all $2$-torsion points would be $\Q$-rational, and in particular, $E$ would have three distinct $2$-isogenies. However, no such isogeny graph can occur for elliptic curves defined over $\Q$ (indeed, it would contradict a theorem of Kenku (\cite[Theorem 4.3]{chiloyan}) that says that there are at most $8$ isogenous curves in a $\Q$-isogeny class; see \cite{chiloyan}, Tables 1-4, which shows the only possibility in this case is a graph of type $R_6$). Hence, we must have that $\Q(E[2])/\Q$ is quadratic. Since $\Q(\sqrt{\Delta_E})$ is the unique quadratic (or trivial) extension contained in $\Q(E[2])$ (see \cite{adel}), it follows that $\Q(E[2])=\Q(\sqrt{\Delta_E})$. Since $\Q(E[3])/\Q$ is biquadratic and contains $\zeta_3$, it follows that $\Q(E[3])=\Q(\sqrt{-3},\sqrt{\Delta_E})$.
	
	Conversely, if $E/\Q$ is an elliptic curve with one $2$-torsion point defined over $\Q$, and $\Im \rho_{E,3}$ is contained in the split Cartan subgroup of $\GL(2,\Z/3\Z)$, then we can find a quadratic twist $E'$ of $E$ such that $\Q(E'[3])=\Q(\sqrt{-3})$, and then a quadratic twist $E''$ of $E'$ by $\Delta_{E'}$ (also a twist of $E$), such that $\Q(E''[3])=\Q(\sqrt{-3},\sqrt{\Delta_{E'}})$. Note that $\Delta_E$, $\Delta_{E'}$, and $\Delta_{E''}$ differ by $6$-th powers, because the curves are quadratic twists of each other, and so $\Q(\sqrt{\Delta_E})=\Q(\sqrt{\Delta_{E'}})=\Q(\sqrt{\Delta_{E''}})$, and $j(E'')=j(E')=j(E)$. Hence 
	$$\Q(E''[2])=\Q(\sqrt{\Delta_{E''}})\subsetneq \Q(\sqrt{-3},\sqrt{\Delta_{E''}}) = \Q(E''[3])=\Q(E''[6]).$$ Thus, every elliptic curve $E/\Q$ with one $2$-torsion point defined over $\Q$ and split Cartan mod $3$ image has a twist $E''$ with the property $\Q(E''[2])\subsetneq \Q(E''[3])=\Q(E''[6]),$ and viceversa, an elliptic curve $E''$ with such property has a $2$-torsion point over $\Q$ and split Cartan image mod $3$. Using Magma \cite{Magma} we have computed the $j$-line of the modular curve that parametrizes elliptic curves with split Cartan image mod $3$ and a $2$-torsion point over $\Q$, and it is given by
$$j(t)=-\left(\frac{(t^3 - 3t^2 - 9t - 9)(t^3 + 3t^2 + 3t - 3)(t^6 + 12t^5 + 81t^4 + 216t^3 + 243t^2 + 108t + 27)}{t(t+1)^2(t+3)^2(t^2+3)^2(t^2+3t+3)}\right)^3.$$
Hence, we conclude that if $E''$ is an elliptic curve with $\Q(E''[2])\subsetneq \Q(E''[3])=\Q(E''[6]),$ then there is some $t\in\Q$ such that $j(E'')=j(t)$ and, if $E/\Q$ is a curve with $j(E)=j(t)$ for some $t$, then there is an appropriate quadratic twist with the desired property. This concludes the proof of the theorem.
\end{proof}

\begin{example}
	We illustrate the case of $(n,m)=(3,6)$ with an example, that is, we are looking for a curve $E''$ with $\Q(E''[2])\subsetneq \Q(E''[3])=\Q(E''[6]).$ First, we evaluate the function $j(t)$ at $t=1$ to obtain $j_1= 9938375/21952$, and we find an elliptic curve $E$ with $j(E)=j_1$, which is given by
	$$E: y^2 + xy = x^3 + 5796284120487x -	9014728680220686983.$$
	This curve satisfies $E(\Q)_\text{tors}\cong \Z/2\Z$ and there are two independent $3$-isogenies. Next, we compute the field of definition of one of the $3$-isogenies: $K=\Q(\sqrt{1137565})$ (this can be accomplished using division polynomials), and find a twist $E'$ of $E$ by $1137565$ which is given by:
	$$E': y^2 + xy + y = x^3 + 4x - 6.$$
	The curve $E'$ satisfies $E'(\Q)_\text{tors}\cong \Z/6\Z$ and $\Delta_{E'}=-2^67^3$. Finally, we find a twist $E''$ of $E'$ by $-7$ which is given by
	$$E'':  y^2 = x^3 + 284445x + 97999902.$$
	The curve $E''$ satisfies $\Q(E''[2])=\Q(\sqrt{-7})$ and $\Q(E''[3])=\Q(E''[6])=\Q(\sqrt{-3},\sqrt{-7})$, as desired.
\end{example}

\begin{remark}
	It is worth pointing out that in \cite[Theorem 8.7]{morrow} Morrow gives an explicit parametrization of elliptic curves $E$ and primes $p\geq 7$ such that $\Q(E[2])\cap \Q(E[p^n])$ is a cyclic cubic field, for some $n\geq 1$.
\end{remark}

\section{Intersections of division fields and $\QQ^{ab}$}\label{sec:abelian_exts}

In order to prove Theorem \ref{thm:main_horizontal} we start by classifying the possible roots of unity in $\QQ(E[p])$. First, we prove that the powers on the prime-to-$p$ roots of unity cannot be very large compared to $p$.

%In this section we explore the possibility of a coincidence of division fields of the form $\QQ(E[p^n])=\QQ(E[q^m])$ for some primes $p<q$.

\begin{prop} \label{prop:roots_in_smaller_prime_div_field}
Let $E/\QQ$ be an elliptic curve, let $p<q$ be primes, let $n,m\geq 1$, and suppose $\QQ(\zeta_{q^m})\subseteq \QQ(E[p^n])$. Then, $m=1$ unless $p=2$ and $q=3$ in which case $m\leq 2$. 
\end{prop} 
\begin{proof} 
Suppose that $\QQ(\zeta_{q^m})\subseteq \QQ(E[p^n])$, for some $m\geq 2$. Then, by Corollary \ref{cor-order_p(p-1)k}, the order of $\Gal(\QQ(E[p^n])/\QQ)$ is divisible by $\varphi(q^m)$, and therefore divisible by $q$. Since $\Gal(\QQ(E[p^n])/\QQ)$ is a subgroup of $\GL(2,\ZZ/p^n\ZZ)$, it follows that $q^{m-1}(q-1)$ divides $p^{4(n-1)+1}(p-1)^2(p+1)$. Since $p\neq q$ are primes, we conclude that $p\equiv \pm 1 \bmod q^{m-1}$, and since $q>p$ we must have $p\equiv -1\bmod q$ and therefore $p=q-1$, and $m=2$. This is only possible if $p=2$ and $q=3$. 
\end{proof} 

Before continuing to the proof of Theorem \ref{thm:main_horizontal} we will need to better understand how large $K_{E}(p) = \QQ(E[p])\cap\QQ^{ab}$ can be and what the structure of $\Gal(K_E(p)/\QQ)$ can be. The size of this field and the possible structure of its Galois group will depend on the image of $\rho_{E,p}$ and as such we will break this section down according to the maximal group that contains $\Im \rho_{E,p}$ (see Proposition \ref{prop:max_groups}).

\subsection{Full Image} From the results of \cite{Duke,Jones} we know that this is in fact the generic case and it turns out to also be the simplest in our context. 
\begin{prop}
Let $E/\QQ$ be an elliptic curve, let $p,q>2$ be distinct odd primes, let $n\geq 1$, and suppose $\rho_{E,p^n}$ is surjective. Then,  the intersection $\QQ(\zeta_{q^m})\cap \QQ(E[p^n])$ is trivial.
\end{prop}
\begin{proof}
Suppose that $\Gal(\QQ(E[p^n])/\QQ)\simeq \GL(2,\ZZ/p^n\ZZ)$. Then, if $p>2$, the commutator subgroup of $\GL(2,\ZZ/p^n\ZZ)$ is $\SL(2,\ZZ/p^n\ZZ)$ (see \cite{serre}), and therefore the largest abelian quotient of $\Gal(\QQ(E[p^n])/\QQ)$ is isomorphic to $(\ZZ/p^n\ZZ)^\times$. Since $\QQ(\zeta_{p^n})\subseteq \QQ(E[p^n])$ by the Weil pairing, it follows that the largest abelian subextension of $L\subseteq \QQ(E[p^n])$ is precisely $L=\QQ(\zeta_{p^n})$. In particular, $\QQ(\zeta_{q^m})\cap \QQ(E[p^n])\subseteq L = \QQ(\zeta_{p^n})$ and therefore, the intersection must be trivial, since $q\neq p$.
\end{proof}

\begin{corollary}\label{cor:surj}
Let $E/\QQ$ be an elliptic curve, $p>2$ a prime, and $n\geq 1$. If $\rho_{E,p^n}$ is surjective then $K_{E}(p) = \QQ(E[p^n]) \cap \QQ^{ab} = \QQ(\zeta_{p^n})$. Further, in this case we have that $\Gal(K_{E}(P)/\QQ) \simeq (\ZZ/p^n\ZZ)^\times.$
\end{corollary}

\subsection{Borel Image}

\begin{defn}\label{defn-borel1} Let $p$ be a prime, and $n\geq 1$. We say that a subgroup $B$ of $\GL(2,\ZZ/p^n\ZZ)$ is Borel if every matrix in $B$ is upper triangular, i.e.,
$$B\subseteq \left\{ \left(\begin{array}{cc} a & b \\ 0 & c\end{array}\right) : a,b,c\in\ZZ/p^n\ZZ,\ a,c\in(\ZZ/p^n\ZZ)^\times \right\}.$$
We say that $B$ is a non-diagonal Borel subgroup if none of the conjugates of $B$ in $\GL(2,\ZZ/p^n\ZZ)$ is formed solely by diagonal matrices. If $B$ is a Borel subgroup, we denote by $B_1$ the subgroup of $B$ formed by those matrices in $B$ whose diagonal coordinates are $1\bmod p^n$, and we denote by $B_d$ the subgroup of $B$ formed by diagonal matrices, i.e.,
$$B_1 = B\cap \left\{ \left(\begin{array}{cc} 1 & b \\ 0 & 1\end{array}\right) : b\in\ZZ/p^n\ZZ \right\}, \text{ and } B_d = B\cap \left\{ \left(\begin{array}{cc} a & 0 \\ 0 & c\end{array}\right) : a,c\in(\ZZ/p^n\ZZ)^\times \right\}.$$
\end{defn}

\begin{lemma}\cite[Lemma 2.2]{lozano}\label{lem-borel} Let $p>2$ be a prime, $n\geq 1$ and let $B\subseteq \GL(2,\ZZ/p^n\ZZ)$ be a Borel subgroup, such that $B$ contains a matrix $g=\left(\begin{array}{cc} a & b \\ 0 & c\end{array}\right)$ with $a\not\equiv c \bmod p$. Let $B'=h^{-1}Bh$ with $h=\left(\begin{array}{cc} 1 & b/(c-a) \\ 0 & 1\end{array}\right)$. Then, $B'\subseteq \GL(2,\ZZ/p^n\ZZ)$ is a Borel subgroup conjugated to $B$ satisfying the following properties:
\begin{enumerate}
\item $B' = B_d'B_1'$, i.e., for every $M\in B'$ there is $U\in B_d'$ and $V\in B_1'$ such that $M=UV$; and
\item $B/[B,B]\simeq B'/[B',B']$ and $[B',B']=B_1'$. 
\end{enumerate}
If follows that $[B,B]=B_1$ and it is a cyclic subgroup of order $p^s$ for some $0\leq s\leq n$.
\end{lemma}

\begin{prop}\label{prop:borel}
Let $E/\QQ$ be an elliptic curve, let $p>2$ be a prime, let $n\geq 1$, and suppose $\Im\rho_{E,p^n}\subseteq \GL(2,\ZZ/p^n\ZZ)$ is a Borel subgroup. Then, the Galois group of the maximal abelian subextension $L\subseteq \Gal(\QQ(E[p^n])/\QQ)$ is a subgroup of $(\ZZ/p^n\ZZ)^\times \times (\ZZ/p^n\ZZ)^\times$, and $L$ contains $\QQ(\zeta_{p^n})$. In particular, if $q\neq p$ is another prime and $\QQ(\zeta_{q^m})\subseteq \QQ(E[p^n])$ for some $m\geq 1$, then $\varphi(q^m)$ is a divisor of $\varphi(p^n)$. If in addition $q>p$, then $m=1$ and $q-1$ is a divisor of $\varphi(p^n)$.
\end{prop}
\begin{proof}
Let $G=\Gal(\QQ(E[p^n])/\QQ)\simeq \Im\rho_{E,p^n}\subseteq \GL(2,\ZZ/p^n\ZZ)$, and suppose $G$ is a Borel subgroup. Then, $G$ contains a matrix $g$ as in Lemma \ref{lem-borel}, because otherwise $\det(G)$ would consist only of square classes, and therefore would not be all of $(\ZZ/p^n\ZZ)^\times$. Thus, by Lemma \ref{lem-borel}, the commutator of $G$ is $G_1$ and $G/G_1\simeq G_d$. Thus, $\Gal(L/\QQ)\simeq G/G_1$, where $L$ is the maximal abelian subextension $L\subseteq \Gal(\QQ(E[p^n])/\QQ)$, is isomorphic to a subgroup of   $(\ZZ/p^n\ZZ)^\times \times (\ZZ/p^n\ZZ)^\times$, as desired.

Now, if $q\neq p$ is another prime and $\QQ(\zeta_{q^m})\subseteq \QQ(E[p^n])$ for some $m\geq 1$, then the compositum $K=\QQ(\zeta_{p^n})\QQ(\zeta_{q^m})$ is contained in $L$. Since the primes are distinct, then $\Gal(K/\QQ)\simeq (\ZZ/p^n\ZZ)^\times \times (\ZZ/q^m\ZZ)^\times$, and since $K\subseteq L$, it follows that $(\ZZ/q^m\ZZ)^\times$ must be a subgroup of $(\ZZ/p^n\ZZ)^\times$. It follows that $\varphi(q^m)$ is a divisor of $\varphi(p^n)$. If in addition we have $q>p$, it follows that $m=1$, and $q-1$ divides $\varphi(p^n)$, as claimed.
\end{proof}

\begin{corollary}\label{cor:borel}
Let $E/\QQ$ be an elliptic curve, $p>2$ a prime, and $n\geq 1$. If $\Im\rho_{E,p^n}$ is  contained in a Borel subgroup of $\GL(2,\ZZ/p\ZZ)$ then the field $K_{E}(p) = \QQ(E[p^n]) \cap \QQ^{ab}$ has $\Gal(K_{E}(P)/\QQ)$ isomorphic to a $(\ZZ/p^n\ZZ)^\times\times C$ where $C$ is a cyclic group of order dividing $\varphi(p^n)$. Thus, $K_p(E)$ is the compositum of $\Q(\zeta_{p^n})$ and a cyclic extension $L/\Q$ with $\Gal(L/\Q)=C$.
\end{corollary}

\subsection{Exceptional Images}

Serre showed that an elliptic curve over $\QQ$ cannot have exceptional image for $p\geq 17$ (see \cite[Lemma 18]{serre2}). Moreover:

\begin{thm}\cite[Theorem 8.1]{lozano0}\label{thm-exceptional} Let $E/\QQ$ be an elliptic curve, and $p\geq 3$ a prime number, such that the image $\overline{G}$ of $\rho_{E,p}$ in $\PGL(E[p])$ is isomorphic to $\overline{G}=A_4$, $S_4$, or $A_5$. Then, $p\leq 13$ and $\overline{G}=S_4$.
\end{thm}

Further, by work of Sutherland and Zywina \cite{Sutherland2, zywina1}, we know that if $3\leq p \leq 13$ and  the image of $\rho_{E,p}$ in $\PGL(E[p])$ is isomorphic to $S_4$, then $p=5$ or $p=13$.

\begin{prop}\cite{Sutherland2, zywina1}\label{prop-sz}  Let $E/\QQ$ be an elliptic curve, and $p\geq 3$ a prime number, such that the image of $\rho_{E,p}$ in $\PGL(E[p])$ is isomorphic to $S_4$. Then, $p=5$, or $p=13$. Moreover, $\Im \rho_{E,p}\subseteq \GL(2,\ZZ/p\ZZ)$ is a conjugate subgroup of
\begin{eqnarray*}
H_5=\left\langle \left(\begin{array}{cc} 1 & 4 \\ 1 & 1\end{array}\right),\left(\begin{array}{cc} 1 & 0 \\ 0 & 2\end{array}\right)\right\rangle & \subseteq & \GL(2,\ZZ/5\ZZ), \ \text{ or }\\ H_{13}=\left\langle \left(\begin{array}{cc} 1 & 12 \\ 1 & 1\end{array}\right),\left(\begin{array}{cc} 1 & 0 \\ 0 & 8\end{array}\right)\right\rangle & \subseteq & \GL(2,\ZZ/13\ZZ).
\end{eqnarray*}\end{prop}
\begin{proof}
See \cite{Sutherland2}, Tables 3 and 4, and \cite{zywina1}, Theorems 1.4 and 1.8.
\end{proof}

\begin{prop}\label{prop-exceptional}
Let $E/\QQ$ be an elliptic curve, let $p>2$ be a  prime, and let $n\geq 1$. Moreover, assume that the image of $\rho_{E,p}$ is an exceptional subgroup (i.e., the projective image of $\Im \rho_{E,p}$ in $\PGL(2,\ZZ/p\ZZ)$ is isomorphic to $S_4$). Then, $\QQ(\zeta_{q^m})\cap \QQ(E[p^n])$ is trivial for any $m\geq 1$ and any prime $q\neq p$.
\end{prop}
\begin{proof}
Suppose that $E$ is a curve as in the statement. Then, by Theorem \ref{thm-exceptional}, we have $p\leq 13$, and by Proposition \ref{prop-sz}, we have $p=5$ or $p=13$, and $G=\Im \rho_{E,p}$ is a conjugate of $H_5$ or $H_{13}$, respectively. A simple computation shows that 
$$H_5/[H_5,H_5]\simeq (\ZZ/5\ZZ)^\times \text{  and  }  H_{13}/[H_{13},H_{13}]\simeq (\ZZ/13\ZZ)^\times$$
and therefore the Galois group of the maximal abelian subextension $L_p$ of $\QQ(E[p])$ is isomorphic to $(\ZZ/p\ZZ)^\times$, for $p=5$, or $13$. Since $\QQ(\zeta_p)\subseteq \QQ(E[p])$, it follows that $L_p=\QQ(\zeta_p)$. In particular, if $q\neq p$, then $\QQ(\zeta_{q^m})\cap \QQ(E[p])\subseteq L_p$ must be trivial. Now, if $\QQ(\zeta_{q^m})\subseteq \QQ(E[p^n])$, and since $[\QQ(E[p^n]):\QQ(E[p])]$ is a power of $p$, it would follow that $\varphi(q^m)$ is itself a power of $p$. Hence, $0\leq m\leq 1$, and $q-1=p^t$ for some $t\geq 1$, which is impossible for $p>2$, unless $m=0$. Thus, $\QQ(\zeta_{q^m})\cap \QQ(E[p^n])$ is trivial for any $q\neq p$ and any $m\geq 1$, as desired. 
\end{proof}

\begin{corollary}
Let $E/\QQ$ be an elliptic curve and  $p>2$ a prime. If $\rho_{E,p}$ is contained in an exceptional group then $K_{E}(p) = \QQ(E[p]) \cap \QQ^{ab} = \QQ(\zeta_{p})$. Thus, in this case we have that $\Gal(K_{E}(p)/\QQ) \simeq (\ZZ/p\ZZ)^\times.$
\end{corollary}

\subsection{Split Cartan} In this subsection we give results in the case when the image is contained in the normalizer of a split Cartan group. We define the split Cartan subgroup of $\GL(2,\Z/p\Z)$ by
$$C_s(p) = \left\{\begin{pmatrix}a&0\\0&b \end{pmatrix}: a,b\in(\ZZ/p\ZZ)^\times \right\}$$ and let $N_s(p) = C_s(p) \cup \begin{pmatrix}0&1\\1&0\end{pmatrix}\cdot C_s(p)$ be its normalizer. Notice that if the image $G$ of $\rho_{E,p}$ is strictly contained in a split Cartan group, then $G$ is abelian, in which case we already know what can happen by our results in Section \ref{sec-abelian}. Thus, we will assume that $G$ is non-abelian.

\begin{thm}\label{prop-main_split} 
If $G$ is a subgroup of $N_s(p)$ such that $\det(G) = (\ZZ/p\ZZ)^\times$, and $G' = [G,G]\neq \{\operatorname{Id}\},$ then $G/G'$ is isomorphic to $(\ZZ/p\ZZ)^\times$, or $(\ZZ/p\ZZ)^\times \times \ZZ/2\ZZ$, or $\ZZ/2(p-1)\ZZ$. Moreover, if $G$ contains an element of zero trace and determinant $-1$, then $G/G'\simeq (\ZZ/p\ZZ)^\times$ or $(\ZZ/p\ZZ)^\times\times \ZZ/2\ZZ$.
\end{thm}

Before we prove Theorem \ref{prop-main_split}, we will need the  following lemmas.

\begin{lemma}\label{lem_square}\label{lem_ker1}
Let $G$ and $G'=[G,G]$ be as in Theorem \ref{prop-main_split}. Then,
\begin{enumerate} 
\item If $M=\begin{pmatrix} a&0\\ 0&b \end{pmatrix}\in G$, then $N=\begin{pmatrix} b&0\\ 0&a \end{pmatrix}\in G$ also, and $M \equiv N \bmod G'$.
\item Let $H$ be a subgroup of the form $\left\{ \begin{pmatrix}a&0\\0&a^{-1} \end{pmatrix}: a\in(\ZZ/p\ZZ)^\times \right\}\cap G$, and let $\overline{H}$ be the image of $H$ in $G/G'$.
Then, $\overline{H}$ is either trivial or isomorphic to $\ZZ/2\ZZ$.
\item If $\overline{H}$ is non-trivial, and there is a matrix in  $G$ of the form $T_a=\begin{pmatrix} a&0\\ 0&a^{-1} \end{pmatrix}$ where $a$ is not a quadratic residue mod $p$, then $\overline{H}$ is generated by $T_a \bmod G'$.
\end{enumerate}
\end{lemma}

\begin{proof}
For (1), suppose that $M=\left(\begin{smallmatrix} a&0 \\ 0&b \end{smallmatrix}\right)\in G$. Since $G' \neq \{ \operatorname{Id}\}$, there must also be an element in $G$ of the form $L=\left(\begin{smallmatrix} 0&c\\ d&0 \end{smallmatrix}\right).$ Therefore,
\[[M,L]=\begin{pmatrix} a&0 \\ 0&b \end{pmatrix} \begin{pmatrix} 0&c\\ d&0 \end{pmatrix} \begin{pmatrix} a&0 \\ 0&b \end{pmatrix}^{-1} \begin{pmatrix} 0&c\\ d&0 \end{pmatrix}^{-1}=\begin{pmatrix} ab^{-1}&0 \\ 0&a^{-1}b \end{pmatrix} \in G'.\]
Thus, $N=M\cdot [M,L]^{-1}$ belongs to $G$ also, and $M\equiv N\bmod G'$, as claimed.

For (2), since $H$ is cyclic, the image $\overline{H}$ is also cyclic. Further, part (1) shows that every element of $\overline{H}$ has order dividing two since 
$$\begin{pmatrix} a&0\\ 0&a^{-1} \end{pmatrix}^2 \equiv \begin{pmatrix} a&0\\ 0&a^{-1} \end{pmatrix} \begin{pmatrix} a&0\\ 0&a^{-1} \end{pmatrix} \equiv  \begin{pmatrix} a&0\\ 0&a^{-1} \end{pmatrix} \begin{pmatrix} a^{-1}&0\\ 0&a \end{pmatrix} \equiv \begin{pmatrix} 1&0\\ 0&1 \end{pmatrix}\bmod G'.$$ Thus, $\overline{H}$ is trivial or cyclic of order 2. This shows (2).

For (3), let $\mathcal{N}_p\subset (\Z/p\Z)^\times$ be the set of quadratic non-residues mod $p$. Let $a\in \mathcal{N}_p$ and put $T_a=\begin{pmatrix} a&0\\ 0&a^{-1} \end{pmatrix}$. Let us suppose that $\overline{H}$ is non-trivial (therefore of order $2$ by part (2)), and $T_a\in G$. Since $H$ is cyclic, $H$ is generated by a matrix $T_d$ for some $d\in (\Z/p\Z)^\times$. Notice that $d\in \mathcal{N}_p$, because there is an odd number $n$ with $T_d^n=T_a$ (and therefore $d^n\equiv a \bmod p$) but this would be impossible if $d$ was a square or $n$ was even. Finally, if we write $n=2k+1$, and since $|\overline{H}|=2$, we have $T_d^2\equiv \operatorname{Id} \bmod\, G'$, and  it follows that $T_a = T_d^n = T_d^{2k}T_d \equiv T_d \bmod G'$, and therefore $\overline{H}$ is generated by the class of $T_a$ as well,  as desired.
\end{proof}

\begin{lemma}\label{lem_minusid}
With notation and assumptions as in Proposition \ref{prop-main_split}, and if we assume that $G$ contains a matrix of the form $A_a=\begin{pmatrix} 0&-a\\ a^{-1}& 0 \end{pmatrix}$ for some $a\in(\ZZ/p\ZZ)^\times$, then 
\begin{enumerate} 
\item Let $\overline{H}$ be the image of subgroup $H = \left\{ \begin{pmatrix}d&0\\0&d^{-1} \end{pmatrix}: d\in(\ZZ/p\ZZ)^\times \right\}\cap G$ in $G/G'$. Then, $\overline{H}$ is trivial.

\item The matrix $-\operatorname{Id}$ is a commutator of $G$, i.e., $-\operatorname{Id}\in G'$. 

\end{enumerate} 
\end{lemma} 
\begin{proof}
Let $G$ be a group as in the statement of Proposition \ref{prop-main_split}. Since $\det(G)=(\ZZ/p\ZZ)^\times=\langle g\rangle$, for some primitive root $g\bmod p$, there is a matrix $M\in G\subseteq N_s(p)$ with $\det(M)=g$. Hence, $M=\begin{pmatrix} gc&0\\ 0&c^{-1} \end{pmatrix}$ or $\begin{pmatrix} 0&gc\\ -c^{-1} &0 \end{pmatrix}$ for some $c\in(\ZZ/p\ZZ)^\times$. In the latter case, we define $M'$ by 
$$M'=M\begin{pmatrix} 0&-a\\ a^{-1}& 0 \end{pmatrix}=\begin{pmatrix} gca^{-1}&0\\ 0& ac^{-1} \end{pmatrix}.$$
In particular,
$$MA_a M^{-1}A_a^{-1}=\begin{pmatrix} gc^2&0\\ 0&(gc^2)^{-1} \end{pmatrix}, \text{ and } M'A_a M'^{-1}A_a^{-1}=\begin{pmatrix} gc^2a^{-2}&0\\ 0&(gc^2a^{-2})^{-1} \end{pmatrix}.$$
Thus, following the notation of Lemma \ref{lem_square}, part (3), the matrices $T_{gc^2}$ or $T_{g(ca^{-1})^2}\in G'\subseteq G$ and $gc^2$ , $g(ca^{-1})^2$ are not squares modulo $p$. By Lemma \ref{lem_square}, either $T_{gc^2}$ or $T_{g(ca^{-1})^2}$ generate $\overline{H}$, but they both belong to $G'$, and it follows that $\overline{H}$ must be trivial. This proves (1).

For (2), we note that $T_{d}^{(p-1)/2} = T_{d^{(p-1)/2}}$, and 
$$(gd^2)^{(p-1)/2} \equiv g^{(p-1)/2}d^{(p-1)} \equiv -1 \bmod p$$
for any $d\in(\ZZ/p\ZZ)^\times$. Thus, $-\operatorname{Id} = T_{gc^2}^{(p-1)/2}= T_{gc^2a^{-2}}^{(p-1)/2}\in G'$, as desired.
\end{proof}

\begin{lemma}\label{lem_ker2}
With notation as in Proposition \ref{prop-main_split}, let $\overline{S}$ be the image of the subgroup $S = G\cap\SL_2(\ZZ/p\ZZ)$ in $G/G'$.
Then, $\overline{S}$ is cyclic with order dividing $2$. 
\end{lemma}

\begin{proof}
First note that $S = \left\{ \begin{pmatrix}a&0\\0&a^{-1} \end{pmatrix}, \begin{pmatrix}0&-a\\a^{-1}&0 \end{pmatrix}: a\in (\ZZ/p\ZZ)^\times \right\}\cap G$. If $S\subseteq H$, with $H$ as in Lemma \ref{lem_ker1}, then $\overline{H}$ is trivial or isomorphic to $\ZZ/2\ZZ$. Otherwise, there is a matrix of the form $\begin{pmatrix}0&-a\\a^{-1}&0 \end{pmatrix}$ in $S\subseteq G$. Hence, Lemma \ref{lem_minusid} implies that $H\subseteq G'$ (i.e., $\overline{H}$ is trivial) and  $-\operatorname{Id}\in G'$. Now suppose that we have any two elements of $S-H$,
$$A = \begin{pmatrix} 0&-a\\ a^{-1}&0\end{pmatrix}\hbox{  and  } B = \begin{pmatrix} 0&-b\\ b^{-1}&0\end{pmatrix}.$$
Then 
$$A^2 = -\operatorname{Id} \in G', \text{ and } AB = \begin{pmatrix} -ab^{-1}&0\\ 0&-a^{-1}b\end{pmatrix}\in H \subseteq G'.$$ Hence $AB\equiv \operatorname{Id}\bmod G'$, and
$$A\equiv A\cdot AB \equiv A^2\cdot B \equiv B \bmod G'.$$
Thus, $\overline{S}$ is of order $2$.
\end{proof}

\begin{lemma}\label{lem_conj}
If $G$ is a subgroup of $N_s(p)$ that contains a matrix $\tau$ of zero trace and determinant $-1$, then there is an element of $\gamma\in C_s(p)$ such that $\gamma\tau \gamma^{-1}$ is one of the following matrices:
\[
\tau_1 = \begin{pmatrix} -1&0\\ 0&1 \end{pmatrix},\hspace{15pt} \tau_2 = \begin{pmatrix} 1&0\\ 0&-1 \end{pmatrix},\hspace{15pt}
\tau_3 = \begin{pmatrix} 0&1\\ 1&0 \end{pmatrix}.
\]
In particular, $G$ is conjugate to a subgroup of $N_s(p)$ that contains one of $\tau_i$, for $i=1$, $2$, or $3$.
\end{lemma}
\begin{proof}
If $\tau\in N_s(p)$ has zero trace and determinant $-1$, then it is of the form $\begin{pmatrix} \pm 1&0\\ 0&\mp 1 \end{pmatrix}$ or $\begin{pmatrix} 0&a\\ a^{-1}&0 \end{pmatrix}$ for some $a\in(\ZZ/p\ZZ)^\times$. If it is the latter, then the matrix $\gamma=\begin{pmatrix} a&0\\ 0&1 \end{pmatrix}\in C_s(p)$ satisfies $\gamma^{-1}\tau \gamma = \tau_3$, as claimed. Finally, since $N_s(p)$ is stable under conjugation by an element $\gamma$ of its subgroup $C_s(p)$, the last claim follows.
\end{proof}

We are finally ready to prove Theorem \ref{prop-main_split}.

\begin{proof}[Proof of Theorem \ref{prop-main_split}]
Suppose that $G$ is a subgroup of $N_s(p)$ such that $\det(G) = (\ZZ/p\ZZ)^\times$,  and $G' = [G,G]\neq \{\operatorname{Id}\}$. Since $G'$ is a subgroup of $\SL_2(\ZZ/p\ZZ)$, we know that $\det$ induces a map $\overline{\det}\colon G/G' \to (\ZZ/p\ZZ)^\times$ that is a surjective homomorphism. Moreover, $G/G'$ is abelian. Thus $(G/G_1)/\ker(\overline{\det}) \simeq (\ZZ/p\ZZ)^\times$. But $\ker(\overline{\det}) = \overline{S}$, where $\overline{S}$ is the image of $G\cap \SL(2,\ZZ/p\ZZ)$ in $G/G'$. From Lemma \ref{lem_ker2} we have $\overline{S}$ is trivial or isomorphic to $\ZZ/2\ZZ$. Hence, $G/G'$ is isomorphic to $(\ZZ/p\ZZ)^\times$, or $(\ZZ/p\ZZ)^\times \times \ZZ/2\ZZ$, or $\ZZ/2(p-1)\ZZ$. 

It remains to discard the possibility that $G/G'\simeq \ZZ/2(p-1)\ZZ$, under the assumption that $G$ contains an element with zero trace and determinant $-1$. If this happens, then $\overline{S}\simeq \ZZ/2\ZZ$. If so, then $G/G'$  would contain an element of order $2(p-1)$, however we will show that the order of every element in $G/G'$ divides $p-1$. Indeed, every element of $N_s(p)$ has order dividing $2(p-1)$, so if $G/G'$ had an element of order $2(p-1)$, then $G$ itself would have an element $M$ of exact order $2(p-1)$ such that $M \bmod G'$ also has order $2(p-1)$. Such an element of $N_s(p)$ must be of the form $M=\begin{pmatrix} 0&gc\\ c^{-1}&0\end{pmatrix}$, with $g$ a primitive root and $c\not\equiv 0\bmod p$, so that $M^2=\begin{pmatrix} g&0\\ 0&g\end{pmatrix}$ has order $p-1$. In particular, $M^{(p-1)}=(M^2)^{(p-1)/2}=-\operatorname{Id}$, and since the order of $M \bmod G'$ is $2(p-1)$ we conclude $-\operatorname{Id}$ is not in $G'$.  We will prove that in fact $-\operatorname{Id}\in G'$, which is a contradiction.

Since $-\operatorname{Id}$ belongs to the center of $\GL(2,\FF_p)$, the element $-\operatorname{Id}$ belongs to $G'=[G,G]$ if and only if $-\operatorname{Id}$  belongs to the commutator of any subgroup of $N_s(p)$ that is conjugate to $G$. Thus, by Lemma \ref{lem_conj}, we can assume that $G$ contains an element $\tau=\tau_i$, for $i=1$, $2$, or $3$. If $\tau=\tau_1$ or $\tau_2$, then $M\tau M^{-1}\tau^{-1}=-\operatorname{Id}\in G'$ and we are done. Otherwise, suppose $\tau=\tau_3$. Then, 
$$M\tau M^{-1}\tau^{-1} = \begin{pmatrix} gc^2&0 \\ 0&(gc^2)^{-1}\end{pmatrix}=T_{gc^2},$$
but since $g$ is a primitive root, then $gc^2$ is not a square, and $T_{gc^2}^{(p-1)/2}=-\operatorname{Id}\in G'$.  Hence, we have reached a contradiction and $G/G'\simeq \ZZ/(2(p-1))\ZZ$ is impossible, which concludes the proof of the theorem.
\end{proof}

\begin{corollary}\label{cor:split}
Let $E/\QQ$ be an elliptic curve and  $p>2$ a prime. If $\rho_{E,p}$ is contained in the normalizer of a split Cartan subgroup of $\GL(2,\ZZ/p\ZZ)$ then $K_{E}(p) = \QQ(E[p]) \cap \QQ^{ab} \subseteq \QQ(\zeta_{p},\sqrt{d})$ for some $d\in\ZZ$. Thus, in this case we have that $\Gal(K_{E}(p)/\QQ)$ is isomorphic to $(\ZZ/p\ZZ)^\times$ of $\ZZ/2\ZZ \times (\ZZ/p\ZZ)^\times$. 
\end{corollary}

%As an immediate consequence we get the following corollary:
%\begin{corollary}
%Let $p$ a prime, $n,m\in\NN$ with $m<n$, and $E$ be an elliptic curve defined over $\QQ$ with $\QQ(E[p^{n}]) = \QQ(E[p^m])$. Then $p=2$, $m=1$, and $n=2$.
%\end{corollary}

\subsection{Non-split Cartan} In this section we give results in the case that the image of the mod $p$ Galois representation is contained a non-abelian subgroup of the normalizer of the non-split Cartan subgroup of $\GL(2,\ZZ/p\ZZ)$. 

Let  $p$ be a fixed prime and let $\varepsilon\in (\Z/p\Z)^\times$ be a fixed quadratic non-residue. We define the non-split Cartan subgroup of $\GL(2,\ZZ/p\ZZ)$ by
$$C_{ns}(p) = \left\{ \begin{pmatrix}a&\varepsilon b\\b&a \end{pmatrix}  
: a,b,\in \ZZ/p\ZZ \hbox{ and }(a,b) \neq (0,0) \right\},$$
and its normalizer in $\GL(2,\ZZ/p\ZZ)$
$$N_{ns}(p) = C_{ns}(p) \cup \begin{pmatrix} -1 & 0 \\ 0 & 1 \end{pmatrix}C_{ns}(p).$$

From \cite[Section 5]{lozano1}, we have that $C_{ns}(p) \simeq \FF_{p^2}^\times$ and $N_{ns}(p) \simeq \langle C_{ns}(p) , c \rangle$ where $c = \begin{pmatrix} -1&0\\0&1  \end{pmatrix}.$  If we fix a matrix $A = \begin{pmatrix} a & \varepsilon b\\ b & a \end{pmatrix}$ such that $C_{ns}(p) = \langle A \rangle$, then $A$ is diagonalizable when considered over the larger field $(\ZZ/p\ZZ)[\sqrt{\varepsilon}]$. That is, $ A = QD Q^{-1}$ with 
$$Q = \begin{pmatrix} \sqrt{\varepsilon} & -\sqrt{\varepsilon}\\ 1&1 \end{pmatrix} \hbox{ and } D  = \begin{pmatrix}a+b\sqrt{\varepsilon}&0\\0&a-b\sqrt{\varepsilon} \end{pmatrix}.$$
A simple computation shows that 
\begin{align}\label{eqn:conj_is_exp}
A^p &= QD^pQ^{-1} = Q\begin{pmatrix} (a+b\sqrt{\varepsilon})^p &0\\0&(a-b\sqrt{\varepsilon})^p \end{pmatrix}Q^{-1 }= Q\begin{pmatrix} (a^p+b^p\sqrt{\varepsilon}^p) &0\\0&a^p-b^p\sqrt{\varepsilon}^p \end{pmatrix}Q^{-1}\\ \nonumber
&= Q\begin{pmatrix} (a-b\sqrt{\varepsilon}) &0\\0&a+b\sqrt{\varepsilon} \end{pmatrix}Q^{-1} = \begin{pmatrix} a & -\varepsilon b\\ -b & a \end{pmatrix} = cAc.
\end{align}
We note that the first equality in the second line above follows from the fact that $$\Gal((\ZZ/p\ZZ)[\sqrt{\varepsilon}]/(\ZZ/p\ZZ))\simeq \ZZ/2\ZZ$$ and is generated by the Frobenius map $x \mapsto x^p$ which maps $\sqrt{\varepsilon}\mapsto -\sqrt{\varepsilon}$ (since $(\sqrt{\varepsilon}^p = \varepsilon^{(p-1)/2}\sqrt{\varepsilon}=-\sqrt{\varepsilon}$ by Euler's criterion, because $\varepsilon$ is a quadratic non-residue). Lastly, we point out that since the map $\det\colon C_{ns}(p) \to (\ZZ/p\ZZ)^\times$ is surjective, and if $\langle A \rangle = C_{ns}(p)$, then $\det(A) = \alpha$ where $\alpha$ is a generator of $(\ZZ/p\ZZ)^\times$. 

%\begin{remark}
%    We choose to use $c$ for the matrix $\left(\begin{smallmatrix}-1&0\\0&1 \end{smallmatrix}\right)$ becuase it should be though of as representing the automorphism corresponding to complex conjugation. 
%\end{remark} 

As in the previous section (the split case), we may assume that $G$, the image of $\rho_{E,p}$ is non-abelian, since we have treated the abelian case separately in Section \ref{sec-abelian}. Thus, we will assume here that $G'=[G,G]$ is not trivial. 

\begin{lemma}\label{lem-nscartan}
	Let $G$ be a non-abelian subgroup of $N_{cs}(p)$. Then,
	\begin{enumerate}
		\item  $G=\langle H, \tau\rangle$, where $H$ is a subgroup of $C_{ns}(p)$ (therefore cyclic) and $\tau$ is any element of $N_{cs}(p)-C_{ns}(p)$. 
		\item If $\tau \in N_{cs}(p)-C_{ns}(p)$, then $\tau^2 \in C_{ns}(p)^{p+1}\cap H$. 
		\item Fix $\tau \in N_{cs}(p)-C_{ns}(p)$. Then, every element $g\in G$ is of the form $g=h$ or $g=h\tau$, for some $h\in H$. In particular, $\# G = 2\cdot \# H$.
		\item Fix $\tau \in N_{cs}(p)-C_{ns}(p)$. If $h\in H$, then $\tau h \tau^{-1} = h^p$.
		\item Suppose $G$ contains an element $\lambda$ of order $2$ with zero trace and determinant $-1$. Then, $\lambda\in N_{cs}(p)-C_{ns}(p)$ and $G = \langle H, \lambda\rangle$. Further, $G\simeq H\rtimes_\varphi \Z/2\Z$ with respect to the map $\varphi\colon \Z/2\Z \to \Aut(H)$ that sends $\lambda$ to $h\mapsto \lambda\cdot h\cdot \lambda^{-1}=h^p$.
	\end{enumerate}
\end{lemma}
\begin{proof}
	Let $G$ be a non-abelian subgroup of $N_{cs}(p)$. Let $H=G\cap C_{ns}(p)\subseteq G$. Since $G$ is non-abelian, and $C_{ns}(p)$ is cyclic abelian, it follows that $H$ is cyclic and $H\neq G$, and there is $\tau \in G-H$, hence $\tau\in N_{ns}(p)-C_{ns}(p)$. Now suppose that $\gamma$ is also an element of $G$ not in $H$. Let $\langle A \rangle = C_{ns}(p)$. Then, $\gamma = A^k c$ and $\tau=A^j c$ for some $j, k\geq 0$ and $c$ as above. Thus, 
	$$\gamma \cdot \tau = (A^k c)(A^j c) = A^k (cA^j c)=A^k\cdot A^{jp} = A^{k+jp} \in C_{ns}(p)\cap G = H$$
where we have used Equation (\ref{eqn:conj_is_exp}),	and so $\gamma = h\cdot \tau^{-1} \in \langle H,\tau\rangle$, where $h'=A^{k+jp}$. Hence, $G=\langle H,\tau\rangle$ as we wanted to show. Moreover, $\tau^2 = (A^jc)(A^jc)=A^jA^{pj}=A^{(p+1)j}\in C_{ns}(p)^{p+1}\cap H$. So $\tau^2 = h''$ and $(h'')^{-1}\tau=\tau^{-1}$. Thus, an arbitrary $\gamma\in G-H$ as above can be written as $\gamma =h'\tau^{-1} = h'(h'')^{-1}\tau = h\tau$ for $h=h'(h'')^{-1}\in H$.  
Finally, let $h=A^k\in H$. Then:
$$\tau h \tau^{-1}=\tau A^k\tau^{-1} =A^j c A^k cA^{-j}=A^{j+kp-j}=A^{kp}=h^p,$$
as desired for (4). 

Finally, suppose that $G$ contains an element $\lambda$ as in part (5). Then, $\lambda$ cannot be in $C_{ns}(p)$ because $C_{ns}(p)$ is cyclic of order $p^2-1$ and contains a unique element of order $2$, namely $-\operatorname{Id}$, whose determinant is $1$. Hence, $\lambda \in N_{cs}(p)-C_{ns}(p)$ and our previous work shows that $G=\langle H,\lambda\rangle$ is of order $2\cdot \# H$. Since $\lambda$ is of order $2$, not in $H$, and $\lambda\cdot h\cdot \lambda^{-1}=h^p$ for all $h\in H$, it follows that $G\simeq H\rtimes_\varphi \Z/2\Z$ as claimed.
\end{proof}

\begin{prop}
Let $p$ be a prime and let $G$ be a subgroup of $N_{ns}(p)$ such that $\det(G) = (\ZZ/p\ZZ)^\times$, such that $G$ contains an element $\lambda$ of order $2$, zero trace, and determinant $-1$, and assume that $G' = [G,G]\neq \{{\rm Id}\}$. Then,  
$G/G'\simeq  \ZZ/2\ZZ\times (\ZZ/p\ZZ)^\times.$
\end{prop}

\begin{proof} Let $G$ be a subgroup of $N_{ns}(p)$ such that $\det(G)\simeq (\Z/p\Z)^\times$. By Lemma \ref{lem-nscartan}, we have $G=\langle H, \tau \rangle$, with $H\subseteq C_{ns}(p)=\langle A \rangle$ and $\tau\in N_{cs}(p)-C_{ns}(p)$. Thus, $H=\langle A^{k_0} \rangle$ for some divisor $k_0$ of $p^2-1$. Let $\tau = A^jc$ for some $j\geq 0$. Note that
	$$[A^{k_0},\tau]=A^{k_0}\tau A^{-k_0}\tau^{-1}=A^{k_0-k_0p}=A^{-(p-1)k_0}\in H^{p-1}.$$
	where we have used Lemma \ref{lem-nscartan}, part (4), and since $H=\langle A^{k_0}\rangle $, it follows that $H^{p-1}=\langle [A^{k_0},\tau]\rangle \subseteq G'=[G,G]$.  Further, we claim that $G'=H^{p-1}$. In order to show this, it suffices to show that $G/H^{p-1}$ is abelian. Indeed, $G=H\rtimes \langle \lambda \rangle$ by Lemma \ref{lem-nscartan}, part (5), and in $G/H^{p-1}$ we have $\lambda\cdot h \cdot \lambda^{-1} =h^p \equiv h \bmod H^{p-1}$, for all $h\in H$. Thus, $\lambda h \equiv  h\lambda \bmod H^{p-1}$, and $G/H^{p-1}$ is abelian. Therefore, $G'=H^{p-1}$. 
	
	Finally, note that $H/H^{p-1}$ injects into $C_{ns}(p)/C_{ns}(p)^{p-1}\simeq (\Z/p\Z)^\times$, where the last isomorphism comes from the fact that $C_{ns}(p)$ is cyclic of order $p^2-1$. Thus, $H/H^{p-1}$ is at most of size $p-1$. Moreover:
	$$G/G' = G/H^{p-1} = (H\rtimes \langle \lambda \rangle)/H^{p-1} \simeq (H/H^{p-1})\times \Z/2\Z.$$
	Further, since $G'\subseteq \SL(2,\Z/p\Z)$, it follows that $\det\colon G/G'\to (\Z/p\Z)^\times$ is also surjective. This implies that $H/H^{p-1}\times \Z/2\Z$ has an element of order $p-1$, and therefore $H/H^{p-1}$ must be of size at least $p-1$. Hence, $H/H^{p-1}\simeq (\Z/p\Z)^\times$ and $G/G'\simeq (\Z/p\Z)^\times \times \Z/2\Z$ as desired.
\end{proof}

\begin{corollary}\label{cor:non-split}
Let $E/\QQ$ be an elliptic curve and  $p>2$ a prime. If $\rho_{E,p}$ is contained in the normalizer of a non-split Cartan subgroup of $\GL(2,\ZZ/p\ZZ)$ then $K_{E}(p) = \QQ(E[p]) \cap \QQ^{ab} \subseteq \QQ(\zeta_{p},\sqrt{d})$ for some $d\in\ZZ$. Thus, in this case we have that $\Gal(K_{E}(P)/\QQ)$ is isomorphic to $(\ZZ/p\ZZ)^\times$ of $\ZZ/2\ZZ \times (\ZZ/p\ZZ)^\times$. 
\end{corollary}

\subsection{Summary of the results of this section}

Before ending this section we give a summary the information in Corollaries \ref{cor:surj}, \ref{cor:borel}, \ref{cor:split}, and \ref{cor:non-split} in a single place. 

\begin{prop}\label{prop:summary}
Let $E/\QQ$ be an elliptic curve, $p>2$ a prime, and $K_{E}(p) = \QQ(E[p])\cap\QQ^{ab}$. Then, 
$$\Gal(K_{E}(p)/\QQ)\simeq (\ZZ/p\ZZ)^\times \times C$$
where $C$ is a cyclic group of order dividing $p-1$. Moreover, $\#C>2$ only when $E$ has a $p$-isogeny (i.e., if the image of $\rho_{E,p}$ is contained in a Borel subgroup). 
\end{prop}

Note that Prop. \ref{prop:summary} provides a proof of parts (2) and (3) of Theorem \ref{thm:main_horizontal}, so in Section \ref{sec:horizontal} we will just need to prove part (1).

\section{The proof of Theorem \ref{thm:main_horizontal}}\label{sec:horizontal}

In an attempt to simplify the proof of Theorem \ref{thm:main_horizontal} we start this section by proving a few lemmas. 

\begin{lemma}\label{lem:3=q}
For every elliptic curve $E/\QQ$ and prime $q>3$, $\QQ(E[3])\cap \Q(\zeta_q)$ is at most quadratic. Thus, $\QQ(E[3]) \neq \QQ(E[q])$.
\end{lemma}

\begin{proof}
	Let $E/\Q$ be an elliptic curve, and let $q>3$ be a prime. From Proposition \ref{prop:summary}, we know that the largest that $K_E(3) = \QQ(E[3]) \cap \QQ^{ab}$ is at most biquadratic, of the form $F(\zeta_3)$ for some quadratic field $F/\Q$. Further, since $q>3$ we have $\QQ(\zeta_3) \cap \QQ(\zeta_q) = \QQ$ and $\Q(\zeta_3)$ is quadratic, so $\QQ(E[3])\cap \QQ(\zeta_q)$ is at most quadratic.

For the second part of the statement, suppose towards a contradiction that $\QQ(E[3]) = \QQ(E[q])$. Then, from the first part of the statement, it follows that $\Q(\zeta_q)\subseteq \Q(E[q])=\Q(E[3])$ must be a quadratic subfield of $\Q(E[3])$. Since the degree of $\QQ(\zeta_q)/\QQ$ is $q-1$, this means that $q = 3$ contradicting the assumption that $q>3$.
\end{proof}

\begin{lemma}\label{lem:4=q}
For every elliptic curve $E/\QQ$ and prime $q>2$, such that $\QQ(E[4])/\Q$ is a non-abelian extension, we have $\QQ(E[4]) \neq \QQ(E[q])$.
\end{lemma}

\begin{proof}
Suppose towards a contradiction that there are an elliptic curve $E/\QQ$ and a prime $q>2$ such that $\QQ(E[4]) = \QQ(E[q])$. In particular, $\Q(\zeta_q)\subseteq \QQ(E[4])$.

First note that the order of elements in $\GL(2,\ZZ/4\ZZ)$ are  $1,2,3,4,$ or $6$. Thus, if $\QQ(\zeta_q)\subset\QQ(E[4])$ and $G=G_4  =  \Im\rho_{E,4}$ it must be that  $G/[G,G]$ has an elements of order $q-1$. Since the order of an element in $G/[G,G]$ divides the order of a representative in $G$, we know that $q-1\in\{1,2,3,4,6\}$ and thus $q\in\{2,3,5,7\}$. Based on the assumptions $q\neq 2$ and so $q = 3$, $5$, or $7$.  

Now, since $\Q(E[4])=\Q(E[q])$, it follows that $\Q(E[4q])$ is also the same field. If we let $G_n = \Im\rho_{E,n}$, then we must have that $\Gal(\Q(E[4q])/\Q)\simeq G_{4q}\subseteq \GL(2,\Z/4q\Z)$, and the natural reduction maps $G_{4q}\to G_4$ and $G_{4q}\to G_q$ are isomorphisms. In addition $\det(G_{4q}) = (\Z/4q\Z)^\times$ and there is a matrix in $G_{4q}$ with zero trace and determinant $-1$ (namely, the image of a complex conjugation via $\rho_{E,4q}\colon G_\Q \to \GL(2,\Z/4\Z)$; see \cite[Remark 3.1.4]{Sutherland2}). Further, we have assumed that $G_4$ and therefore $G_{4q}$ are non-abelian.  
A search for such subgroups of $\GL(2,\ZZ/4q\ZZ)$ yields that there are none for $q=5$ or $7$ so $q$ must be 3. The search for subgroups in $\GL(2,\ZZ/12\ZZ)$ yields two possible maximal groups, call them $H_1$ and $H_2$. If we let $\pi_3\colon\GL(2,\ZZ/12\ZZ) \to \GL(2,\ZZ/3\ZZ)$ and $\pi_4\colon\GL(2,\ZZ/12\ZZ) \to \GL(2,\ZZ/4\ZZ)$, the we see that $\pi_3(H_1) = \pi_3(H_2) = N_s(3)$ and 
$$\pi_4(H_1) = \left\langle \begin{pmatrix} 3&3 \\ 0&1 \end{pmatrix}, \begin{pmatrix} 1&3 \\ 2&1 \end{pmatrix} \right\rangle\hbox{ and }\pi_4(H_2) = \left\langle \begin{pmatrix} 1&1 \\ 0&3 \end{pmatrix}, \begin{pmatrix} 1&3 \\ 2&1 \end{pmatrix} \right\rangle.$$
Moreover, we have $H_i \simeq \pi_3(H_i) \simeq \pi_4(H_i)\simeq D_4$, for $i=1,2$, and a computation of the genus of the modular curves $X_{H_i}$ yields that both have genus $9$. Determining the rational points on $X_{H_i}$ for $i=1$ or $2$ would be very difficult, so instead consider the subgroup $\widetilde{H}\subseteq \GL(2,\Z/12\Z)$ such that $\pi_3(\widetilde{H})=N_s(3)$ and 
$$\pi_4(\widetilde{H})=\left\langle \begin{pmatrix} -1&0 \\ 0&-1 \end{pmatrix},\begin{pmatrix} 3&3 \\ 0&1 \end{pmatrix}, \begin{pmatrix} 1&3 \\ 2&1 \end{pmatrix} \right\rangle = \left\langle \begin{pmatrix} -1&0 \\ 0&-1 \end{pmatrix},\begin{pmatrix} 1&1 \\ 0&3 \end{pmatrix}, \begin{pmatrix} 1&3 \\ 2&1 \end{pmatrix} \right\rangle.$$ but we do not require that $\pi_3$ and $\pi_4$ are isomorphisms on $\widetilde{H}$. We note that the groups $\pi_4(H_1)$ and $\pi_4(H_2)$ are, respectively, the groups $G_{10d}$ and $G_{10b}$ in the notation of \cite{RZB}, and $\pi_4(\widetilde{H})$ is the group $G_{10}$.

Let $X_s^+(3)$ and $X_{10}$ be the modular curves that parametrize elliptic curves, respectively, with mod $3$ image conjugate to $N_s(3)$, and with mod $4$ image conjugate to $\pi_4(\widetilde{H})$ (or equivalently, conjugate to $G_{10}$). Both $X_s^{+}(3)$ and $X_{10}$ are curves of genus $0$, and the $j$-invariants of such curves are given by rational functions $j(t)$ and $j'(s)$, respectively. Above we have shown that an elliptic curve $E/\Q$ with $\Q(E[4])=\Q(E[3])$ would satisfy $j(E)=j(t_0)=j'(s_0)$ for some rational numbers $t_0$, and $s_0$. Thus, the point $(t_0,s_0)$ would satisfy the equation $j(t)=j'(s)$, that is
$$j(s)=\frac{1728\cdot (s^2-1/3)^3}{s^2(s^2+1)^2}=\frac{27(t+1)^3(t-3)^3}{t^3}=j'(t).$$
In particular:
$$(s(s^2+1))^2=s^2(s^2+1)^2 = \frac{4^3(s^2-1/3)^3t^3}{(t+1)^3(t-3)^3}=\left(\frac{4(s^2-1/3)t}{(t+1)(t-3)}\right)^3,$$
and so $s(s^2+1)$ is a perfect cube, say $s(s^2+1)=T^3$ or in projective coordinates $C: S^3+SU^2=T^3$. An affine patch of this curve is $A:1+y^2=x^3$, which is the elliptic curve with Cremona label \href{https://www.lmfdb.org/EllipticCurve/Q/144a1/}{\texttt{144a1}} and has $A(\Q)=\{\mathcal{O},(1,0)\}$. Thus, the only points on $C$ are $[T,U,S]=[0,1,0]$ and $[1,0,1]$. Thus, $s=S/U=0/1=0$ or $s=S/U=1/0$ is undefined, so $s=0$ is the only possibility. But if $s=0$, then $j(s)$ is undefined. Hence, there is no such elliptic curve $E/\Q$ with $\Q(E[4])=\Q(E[3])$. 
\end{proof}

\begin{prop}\label{prop:p=q}
Let $E/\QQ$ be an elliptic curve and $p<q$ distinct primes in $\ZZ$ such that $\QQ(E[p]) = \QQ(E[q])$. Then, it must be that $p=2$ and $q=3$.
\end{prop}

\begin{proof}
Let $K_E(p) = \QQ(E[p])\cap\QQ^{ab}$ and $K_E(q) = \QQ(E[q])\cap\QQ^{ab}$. If $\QQ(E[p]) = \QQ(E[q])$, then $K_E(p) = K_E(q)$ and $\Gal(K_E(p)/\QQ) = \Gal(K_E(q)/\QQ)$. Let $G = \Gal(K_E(p)/\QQ) = \Gal(K_E(q)/\QQ)$.

Suppose that $p>2$. Applying Proposition \ref{prop:summary} we get that $G \simeq (\ZZ/p\ZZ)^\times\times C \simeq (\ZZ/q\ZZ)^\times\times C'$ where $\#C$ divides $p-1$ and $\#C'$ divides $q-1$. Next, since $\QQ(\zeta_q)\cap \QQ(\zeta_p) = \QQ$, it must be that $G$ contains a subgroup isomorphic to $(\ZZ/p\ZZ)^\times \times(\ZZ/q\ZZ)^\times$. The only way that this can happen is if $p-1$ divides $q-1$ and $q-1$ divides $p-1$, that is $p=q$. Therefore, $p$ cannot be greater than $2$.

Next suppose $p=2$. From Theorem \ref{thm:abelian} we can assume that $\QQ(E[2])/\QQ$ is a non-abelian extension. In this case, $\Gal(\QQ(E[2])/\QQ)\simeq S_3$ and $\Gal(K_E(2)/\QQ)\simeq \ZZ/2\ZZ$. Therefore, $\QQ(\zeta_q)/\QQ$ is a quadratic extension and the only possibility is $q=3$. 
\end{proof}

Before we prove Theorem \ref{thm:main_horizontal} we need one more lemma and a simplifying remark.

\begin{lemma}\label{lem:p^n=9}
There are no elliptic curves $E/\QQ$ and $n\geq 1$ such that $\QQ(E[2^n]) = \QQ(E[9])$.
\end{lemma}

\begin{proof}
Suppose $E/\QQ$ is an elliptic curve and $n\geq 1$ such that $\QQ(E[2^n]) = \QQ(E[9])$. By Theorem \ref{thm:abelian}, it cannot happen if $\Q(E[9])$ is abelian over $\Q$, so we shall assume that $\Q(E[9])/\Q$ is non-abelian. A Magma search on subgroups $G\subseteq \GL(2,\Z/9\Z)$ shows that if $G$ is a non-abelian subgroup with full determinant map and an element corresponding to complex conjugation, then $G/[G,G]$ is isomorphic to one group in the set
$$S = \{ \ZZ/6\ZZ, \ZZ/6\ZZ\times \ZZ/2\ZZ, \ZZ/6\ZZ\times \ZZ/3\ZZ, \ZZ/6\ZZ\times \ZZ/6\ZZ\}.$$
Since $\QQ(\zeta_{2^n}) \cap \QQ(\zeta_9) = \QQ$, if $\QQ(\zeta_{2^n})\subseteq \QQ(E[9])$ and $n\geq 2$ then it follows that $(\Z/2\Z \times (\Z/2\Z)^{n-2})\times \Z/6\Z$ must be a subgroup of $G/[G,G]$. Thus, $n\leq 2$.

If $n =1$, then  $\QQ(\zeta_9) \subseteq \QQ(E[9]) = \QQ(E[2])$ and since $[\QQ(E[2]) :\QQ]\leq  6$, it must be that $\QQ(E[2]) = \QQ(\zeta_9)$ which is a contradiction because we have assumed $\Q(E[9])$ was non-abelian. 

Suppose next that $n=2$. In this case a Magma search on subgroups $G\subseteq \GL(2,\Z/4\Z)$ shows that if $G$ is a non-abelian subgroup with full determinant map and an element corresponding to complex conjugation, then $G/[G,G]$ is isomorphic to one group in the set
$$T = \{ \ZZ/2\ZZ, (\ZZ/2\ZZ)^2, (\ZZ/2\ZZ)^3, \ZZ/2\ZZ\times\ZZ/4\ZZ, \ZZ/6\ZZ, \ZZ/2\ZZ\times\ZZ/6\ZZ \}.$$
Comparing the lists $S$ and $T$, and using the fact that $\QQ(i,\zeta_{9})\subseteq \QQ(E[9]) = \QQ(E[4])$, we see that the only possibility is that, if we write $G_E(4) =\Im\rho_{E,4}$ and $G_E(9) = \Im\rho_{E,9}$, then 
$$G_E(4)/[G_E(4),G_E(4)] \simeq G_E(9)/[G_E(9),G_E(9)] \simeq \ZZ/2\ZZ\times\ZZ/6\ZZ.$$
Let $\pi_2\colon \GL(2,\ZZ/4\ZZ)\to\GL(2,\ZZ/2\ZZ)$ and $\pi_3\colon \GL(2,\ZZ/9\ZZ)\to\GL(2,\ZZ/3\ZZ)$ be the standard reductions maps and let 
\begin{align*}
C_{ns}(2) &= \left\langle \begin{pmatrix} 0&1\\ 1&1 \end{pmatrix} \right\rangle \subseteq \GL(2,\ZZ/2\ZZ) \\
B(3) &= \left\langle  \begin{pmatrix} 1&1\\ 0&1 \end{pmatrix}, \begin{pmatrix} 2&0\\ 0&1 \end{pmatrix}, \begin{pmatrix} 1&0\\ 0&2 \end{pmatrix} \right\rangle \subseteq \GL(2,\ZZ/3\ZZ) \\
N_{ns}(3) &= \left\langle \begin{pmatrix} 1&0\\ 0&2 \end{pmatrix}, \begin{pmatrix} 2&1\\ 2&2 \end{pmatrix} \right\rangle \subseteq \GL(2,\ZZ/3\ZZ).
\end{align*}
Then, a Magma search among subgroups of $\GL(2,\Z/4\Z)$ and $\GL(2,\Z/9\Z)$ shows that $G_E(4)$ must be a subgroup of $\pi_2^{-1}(C_{ns}(2))$ and $G_E(9)$ is a subgroup of either $\pi_3^{-1}(B(3))$ or $\pi_3^{-1}(N_{ns}(3))$. Checking the subgroups of each of the possibilities, we see that there are no subgroups of $\pi_2^{-1}(C_{ns}(2))$, with full determinant and an element that has determinant $-1$ and trace $0$, that are isomorphic to a subgroup of $\pi_3^{-1}(B(3))$ or $\pi_3^{-1}(N_{ns}(3))$ and thus it is not possible for $\QQ(E[9]) = \QQ(E[4])$. 
\end{proof}

\begin{remark}\label{rmk:reduction}
Before starting the proof of Theorem \ref{thm:main_horizontal} we notice that Proposition \ref{prop:roots_in_smaller_prime_div_field} and  Lemma \ref{lem:p^n=9} allows us to reduce the problem considerably. If $E/\QQ$, $p,q,m,$ and $n$ are as in the statement of Theorem \ref{thm:main_horizontal}, then we must have that $\QQ(\zeta_{q^m})\subseteq \QQ(E[p^n])$. But by Proposition \ref{prop:roots_in_smaller_prime_div_field} this means that either $m=1$ or $m=2$, and $p=2$ and $q=3$, but $\Q(E[9])=\Q(E[2^n])$ cannot occur by Lemma \ref{lem:p^n=9}. Thus, we must have $m=1$.
\end{remark}

\begin{proof}[Proof of Theorem \ref{thm:main_horizontal}]
Proposition \ref{prop:summary} implies parts (2) and (3) of the theorem, so it remains to show (1).  From Remark \ref{rmk:reduction},  we only have to consider the case of $m=1$. That is, we suppose that $E/\QQ$ is an elliptic curve, $p<q$ primes in $\ZZ$, and $n\geq 1$ such that $\QQ(E[p^n]) = \QQ(E[q])$, and we aim to show that $p^n =2$ and $q=3$. By Theorem \ref{thm:abelian}, we may assume $\Q(E[q])/\Q$ is non-abelian.

Le $q>p\geq 2$.  Since $\QQ(E[p^n]) = \QQ(E[q])$, it must be that $\QQ(\zeta_{p^n})\subseteq \QQ(E[q])$. From Proposition \ref{prop:summary} and the work supporting its proof, and since $q\geq 3$,  if $K_E(q) = \QQ(E[q])\cap\QQ^{ab}$, then
$$\Gal(K_{E}(q)/\QQ)\simeq \begin{cases}
 (\ZZ/q\ZZ)^\times & E\hbox{ is surjective or exceptional at $q$,}\\
 (\ZZ/q\ZZ)^\times \times \ZZ/2\ZZ & E\hbox{ is Cartan at $q$,}\\
 (\ZZ/q\ZZ)^\times \times C & E\hbox{ is Borel at $q$,}\\
 \end{cases}
$$
where $C$ is a cyclic group of order dividing $q-1$. Again using the fact that $\QQ(\zeta_q)$ and $\QQ(\zeta_{p^n})$ intersect trivially, it follows that $\Gal(\QQ(\zeta_{p^n})/\QQ) \simeq C$ for some cyclic group of order dividing $q-1$. Moreover, if $E$ is not Borel at $q$, then $C\simeq \ZZ/2\ZZ$. In the case that $E$ does not have a $q$-isogeny, the only prime-powered cyclotomic fields that are trivial or quadratic fields are $p^n = 2,3, 4$. So in this case we would have that $\Q(E[2])=\Q(E[q])$ (and therefore $q=3$ is the only possibility), or $\QQ(E[3]) = \QQ(E[q])$, or $\QQ(E[4]) = \QQ(E[q])$, but Lemmas \ref{lem:3=q} and \ref{lem:4=q} show the latter two cases cannot happen.

Suppose then that $E$ is Borel at $q$. First note that if $p=2$ and $n>2$, then $\Gal(\Q(\zeta_{p^n})/\Q)$ is not cyclic and therefore we reach a contradiction because $C$ is cyclic (hence the only possibilities would be $\Q(E[2])=\Q(E[q])$ or $\Q(E[4])=\Q(E[q])$ which have been already discussed above and lead to $\Q(E[2])=\Q(E[3])$ which we discuss below). Thus, assume $p>2$. Since $E$ is Borel at $q$, it follows that $[\Q(E[q]):\Q]$ is a divisor of $(q-1)^2\cdot q$. We distinguish cases according to the type of image modulo $p$. Let $G_p = \Im\rho_{E,p}\subseteq \GL(2,\Z/p\Z)$. 
\begin{itemize}
	\item Suppose $\rho_{E,p}$ is surjective, so that $G_p=\GL(2,\Z/p\Z)$. Then, $[\Q(E[p^n]):\Q]$ is a divisor of $[\Q(E[p]):\Q]\cdot p^k=(p^2-1)(p^2-p)p^k$ for some $k\geq 1$ (this follows from the fact that the kernel of the map $\GL(2,\Z/p^n\Z)\to \GL(2,\Z/p^{n-1}\Z)$ is of size $p^4$, and $\# \GL(2,\Z/p\Z)=(p^2-1)(p^2-p)$). If $\Q(E[p^n])=\Q(E[q])$, and $q$ divides $d_q=[\Q(E[q]):\Q]$, then $q$ divides $(p^2-1)(p^2-p)p^k = (p-1)^2(p+1)p^{k+1}$, but $p<q$ so this is impossible. This would imply that $d_q$ is a divisor of $(q-1)^2$, and $\Q(E[q])/\Q$ is in fact abelian (the mod $q$ image would be contained in a split Cartan subgroup). However we have assumed  $\Q(E[q])/\Q$  is non-abelian.
	\item Suppose the image of $\rho_{E,p}$ is exceptional. Then, Prop. \ref{prop-exceptional} shows that  $\QQ(\zeta_{q^m})\cap \QQ(E[p^n])$ is trivial for any $m\geq 1$ and any prime $q\neq p$, and therefore $\Q(E[q])=\Q(E[p^n])$ would be impossible.
	\item Suppose the image of $\rho_{E,p}$ is contained in the normalizer of a (split or non-split) Cartan subgroup of $\GL(2,\Z/p\Z)$. Then, $[\Q(E[p^n]):\Q]$ is a divisor of $2(p^2-1)p^k$ or $2(p-1)^2p^k$, for some $k\geq 1$ (because $\# C_{ns}(p)=p^2-1$ and $\# C_s(p)=(p-1)^2$). Thus, if $\Q(E[p^n])=\Q(E[q])$, and $q$ divides $d_q=[\Q(E[q]):\Q]$, then $q$ divides $2(p+1)(p-1)^3p^k$, but $p<q$ so this is impossible. As before, this would imply that $\Q(E[q])/\Q$ is abelian, a contradiction. 
\end{itemize}
Hence, we have reached a contradiction in every case, and therefore $\Q(E[p^n])=\Q(E[q])$ is impossible.

All that is left to complete the proof of Theorem \ref{thm:main_horizontal} is to parametrize all the elliptic curves that have $\QQ(E[2]) = \QQ(E[3])$. To do this we search $\GL(2,\ZZ/6\ZZ)$ for subgroups that have surjective determinant maps and such that the reductions maps $\pi_2\colon\GL(2,\ZZ/6\ZZ) \to \GL(2,\ZZ/2\ZZ)$ and $\pi_3\colon\GL(2,\ZZ/6\ZZ) \to \GL(2,\ZZ/3\ZZ)$ induce isomorphisms when restricted to the given subgroup. The search yields two possibilities for $\Im\rho_{E,6}$, namely 
$$H_1 = \left\langle \begin{pmatrix}5&5\\ 0&1  \end{pmatrix},\begin{pmatrix}2&5\\ 1&3  \end{pmatrix} \right\rangle \hbox{ and } H_2 = \left\langle \begin{pmatrix} 1&1 \\ 0&5  \end{pmatrix},\begin{pmatrix}2&5\\ 1&3  \end{pmatrix} \right\rangle.$$
Elliptic curves with $\Im\rho_{E,6}$ in $H_1$ have a rational point of order $3$, while elliptic curves with $\Im\rho_{E,6}$ in $H_2$ are a twist of the previous curves a curve by $-3$. Checking the genus of the corresponding modular curves $X_{H_i}$ using code from \cite{SZ16} we see that they are both genus $0$ and since we have seen an example of this exact image, we know that $X_{H_1}$ (and $X_{H_2}$) must have (infinitely many) rational points. Computing a model for $X_{H_1}$, we get exactly the elliptic curve over $\QQ(t)$ that is in the statement of the theorem, and the twist by $-3$ produces the parametrization of elliptic curves with image $H_2$. This completes the proof of the theorem.
\end{proof}

\section{Proofs of Theorem \ref{thm-whatpairs} and Corollary  \ref{thm-whatabout3}}\label{sec-whatpairs} 

\subsection{Proof of Theorem \ref{thm-whatpairs}} In this section we first provide a proof of Theorem \ref{thm-whatpairs}. Let $2\leq m < n \leq 10$, and suppose that $E/\Q$ is an elliptic curve such that $\Q(E[m])=\Q(E[n])$. We have already seen that $(m,n)\in \{(2,3),(2,4),(2,6),(3,6)\}$ are possible, in Theorems \ref{thm:main_vertical} and \ref{thm:main_horizontal}  (note that $\Q(E[2])=\Q(E[3])$ also implies that $\Q(E[2])=\Q(E[6])$). Moreover, Theorem \ref{thm:main_horizontal} shows that if $\Q(E[p^a])=\Q(E[q^b])$, then $p^a=2$ and $q^b=3$. Hence, the pairs
$$\{(3,4),(2,5),(3,5),(4,5),(2,7),(3,7),(4,7),(5,7),(3,8),(5,8),(7,8),(2,9),(4,9),(5,9),(7,9),(8,9)\}$$
do not occur. Theorem \ref{thm:main_vertical} says $(4,8)$ and $(3,9)$ do not occur, and also note that $(2,8)$ would imply $(4,8)$, so $(2,8)$ does not occur either.

In order to rule out the pairs in the set
$$L=\{(5,6),(2,10),(3,10),(4,10),(6,10),(7,10),(9,10) \}$$
we have used Magma to compute all the possible subgroups $G_m\subseteq \GL(2,\Z/m\Z)$ and $G_n\subseteq \GL(2,\Z/n\Z)$ which could correspond to images of $\rho_{E,m}$ and $\rho_{E,n}$ respectively. In particular, we find all subgroups with full determinant, and such that they contain an element of determinant $-1$ and zero trace, and then checked that there is no possible isomorphism $G_m\cong G_n$. Therefore $(m,n)\in L$ is impossible. 

Similarly, when $(m,n)= (8,10)$ we have computed all possibilities for $G_m$ and $G_n$, and in this case there are pairs such that $G_m\cong G_n$, but in all such cases $G_n$ is an abelian group. In particular, $\Q(E[10])/\Q$ would be abelian, but the main result of \cite{LR+GJ} shows that this is impossible. 

Finally, in the case of $(m,n)=(6,7)$, we have computed all the possible pairs $G_m\cong G_n$ but, in addition, checked that $G_7$ is none of the subgroups described in \cite{Sutherland2}, which correspond to one of the possible mod-$7$ images that occur for elliptic curves over $\Q$. Thus, $(6,7)$ cannot occur.

In summary we have shown that the only possibilities are
$$(m,n)\in \{(2,3),(2,4),(2,6),(3,6),(4,6),(6,8),(6,9),(5,10)\},$$
as claimed. This concludes the proof of Theorem \ref{thm-whatpairs}.

\subsection{Proof of Corollary \ref{thm-whatabout3}} 
Next we provide a proof of Corollary \ref{thm-whatabout3}. Let $p$ be a prime, and let $m\geq 2$ be an integer divisible by $q^n$ for some odd prime $q$ and $n\geq 1$, such that $\varphi(q^n)$ does not divide $p-1$. Let us suppose for a contradiction that $\Q(E[p])=\Q(E[m])$. Since $q^n$ divides $m$, then $\Q(\zeta_{q^n})\subseteq \Q(E[q^n])\subseteq \Q(E[m])$, and therefore $\Q(\zeta_{q^n})\subseteq \Q(E[p])$ as well. By Prop. \ref{prop:summary} we must have that $\varphi(q^n)$ is a divisor of $p-1$, which contradicts our hypothesis. This concludes the proof of the theorem.

\end{document}